\documentclass[a4paper]{article}
\usepackage{etex}
\usepackage{enumerate}
\usepackage[english]{babel}
\usepackage{graphicx}
\usepackage{multicol}
\usepackage[bottom]{footmisc}
\usepackage[export]{adjustbox}
\usepackage{functan}
\usepackage[T1]{fontenc}
\usepackage{cancel}
\usepackage{calligra}
\usepackage{colortbl}
\usepackage{multirow}
\usepackage{enumerate}
\usepackage{varioref}
\usepackage{booktabs}
\usepackage{amscd}
\usepackage{color}
\usepackage{colortbl}
\usepackage{pstricks}
\usepackage{pst-all}
\usepackage{mparhack}
\usepackage{amssymb}
\usepackage{amsmath}
\usepackage{dsfont}
\usepackage{mathrsfs}
\usepackage{amsfonts}
\usepackage{amssymb}
\usepackage[mathcal]{euscript}
\usepackage[all,cmtip]{xy}
\usepackage{amssymb}
\usepackage{amsthm}

\theoremstyle{definition}
\newtheorem{defn}{Definition}[section]
\newtheorem{ex}[defn]{Example}
\newtheorem{rem}[defn]{Remark}

\theoremstyle{plain}

\newtheorem{thm}[defn]{Theorem}
\newtheorem{prop}[defn]{Proposition}
\newtheorem{lem}[defn]{Lemma}
\newtheorem{cor}[defn]{Corollary}
\numberwithin{equation}{section}

\def \alt96 {`}

\def \R {\mathds{R}}

\def \dsy {\displaystyle}

\def \loc {\mathrm{loc}}

\def \N {\mathds{N}}

\def \d {\mathrm{d}}

\def \longto {\longrightarrow}

\begin{document}
 \author{Stefano Biagi, Alessandro Calamai and Francesca Papalini}
 \title{Existence results for
  boundary value problems associated with 
  singular strongly nonlinear equations}
\maketitle
\begin{abstract}
We consider a strongly nonlinear differential equation of the following 
general type
$$(\Phi(a(t,x(t)) \, x'(t)))'= f(t,x(t),x'(t)), \quad \text{a.e.\, on $[0,T]$}$$
where $f$ is a Carath\'edory function, $\Phi$ is a strictly increasing ho\-meo\-mor\-phism 
(the $\Phi$-Laplacian operator)
and the function $a$ is continuous and non-negative.
We assume that $a(t,x)$ is bounded from below by a non-negative function $h(t)$, independent of $x$ and such that
$1/h \in L^p(0,T)$ for some $p> 1$,
and we require a weak growth condition of Wintner-Nagumo type.
Under these assumptions, we prove existence results for the Dirichlet problem associated to the above equation, as well as for different boundary conditions.
Our approach combines fixed point tech\-ni\-ques and the upper/lower solutions method.
\end{abstract}
\section{Introduction}
 Recently many papers have been devoted to the study of boundary va\-lue problems
 (BVPs for short) associated to 
 nonlinear ODEs involving 
 the so-called \textit{$\Phi$-Laplace operator}
 (see, e.g., \cite{cpr,CabadaPo,CabadaPo2,KFA}). 
 Namely, ODEs of the type 
 \begin{equation} \label{eq:intro1}
  \big(\Phi(x')\big)' = f(t,x,x'), 
 \end{equation}
 where $f$ is a Carath\'edory function and 
 $\Phi : \R \to \R$ is a strictly increasing homeomorphism such that $\Phi(0)=0$.

 The class of $\Phi$-Laplacian operators includes as a special case the classical $r$-Laplacian
 $\Phi(y):=y|y|^{r-2}$, with $r>1$.
 Such operators arise in some mo\-dels, e.g. in non-Newtonian
 fluid theory, diffusion of flows in porous media, nonlinear elasticity
 and theory of capillary surfaces.
 Other models
 (for e\-xam\-ple re\-action\--dif\-fu\-sion equations with non-constant diffusivity and
 porous media equations) lead to consider
 mixed differential operators, that is, dif\-fe\-ren\-tial equations of the type
 \begin{equation} \label{eq:intro2}
  \big(a(x)\,\Phi(x')\big)' = f(t,x,x'), 
 \end{equation}
 where $a$ is a continuous positive function (see, e.g., \cite{cmp1}).
 Furthermore, several papers have been devoted to the case of singular
 or non-surjective operators (see \cite{bm08,ca,FerrPap}).
 Usually, these existence results stem from a combination of fixed point techniques 
 with the upper and lower solutions method.
 In this context, an important tool to get a priori bounds 
 for the derivatives of the solutions is a Nagumo-type growth condition on the function $f$.
 Let us observe that, when in the differential operator is present the nonlinear term $a$,
 some assumptions are required on 
 the differential operator $\Phi$, 
 which in general is assumed to be homogeneous, or having at most linear growth at infinity.

 More recently, in collaboration with Cristina Marcelli, we considered two dif\-fe\-rent 
 generalizations of equation
 \eqref{eq:intro2}.
 In the paper \cite{MaPa2017}, it is investigated the case in which the function $a$
 may depend also on $t$.
 More precisely,
 the authors obtain existence results for general
 boundary value problems associated with
 the equation
 \begin{equation*}  \label{eq:mp}
  \big(a(t,x(t))\,\Phi(x'(t))\big)'= f(t,x(t),x'(t)), \quad \text{a.e.\,on $I:= [0,T]$}
 \end{equation*}
 where $a$ continuous and  positive, and
 assuming a weak form of Wintner-Nagumo growth condition.
 Namely,
 \begin{equation} \label{ip:psi2}
  \big|f(t,x,y)\big| \leq 
  \psi\Big(a(t,x)\,|\Phi(y)|\Big)\cdot\Big(\ell(t) + \nu(t)\,|y|^{\frac{s-1}{s}}\Big), 
 \end{equation}
 where $\nu \in L^s(I)$ (for some $s > 1$), $\ell\in  L^1(I)$
 and $\psi:(0,\infty)\to(0,\infty)$ is a measurable function such that
 $1/\psi\in L^1_{\loc}(0,\infty)$ and
 $$ \int_1^{\infty}
 \frac{\d s}{\psi(s)} = \infty. $$
 This assumption
 is weaker than other Nagumo-type conditions previously con\-si\-de\-red,
 and allows to consider a very general operator $\Phi$,
 which can be only strictly increasing,
 not necessarily homogeneous, nor having polynomial growth.
 Let us also observe that the same equation
 \begin{equation*}
  \big(a(t,x)\,\Phi(x')\big)' = f(t,x,x')
 \end{equation*}
 was studied in \cite{Ma2012,Ma2013} to obtain heteroclinic solutions on the real line.

 On the other hand, in \cite{CaMaPa} \textit{possibly singular equations}
 are considered,
 including a non-autonomous differential operator which has an explicit 
 de\-pen\-dence 
 on $t$ inside $\Phi$. Namely
 \begin{equation} \label{eq:intro3} 
  \Big(\Phi\big(k(t)\,x'(t)\big)\Big)'=f(t,x(t),x'(t)), \text{a.e.\,on $I$}
 \end{equation}
 where the function $k$ is allowed to vanish in a set having null measure, 
 so that equation \eqref{eq:intro3} can become singular.
 \cite{CaMaPa} assumed $1/k \in L^p(I)$ and look for solutions in the space $W^{1,p}(I)$,
 rather than $C^{1}(I,\R)$.
 According our knowledge, very few papers have been devoted to this type of equations, and just for a 
 restricted class of nonlinearities $f$
 (see \cite{liu12,liuyang}).
  \medskip
 
 In this paper we tackle a further generalization of equation \eqref{eq:intro3}, 
 al\-low\-ing also a dependence on $x$ inside $\Phi$.
 More in detail, we consider the BVP
 \begin{equation} \label{eq.mainBVP-intro}
	 \begin{cases}
   		\dsy\Big(\Phi\big(a(t,x(t))\,x'(t)\big)\Big)' = f(t,x(t),x'(t)), & \text{a.e.\,on $I$}, \\[0.2cm]
   		x(0) = \nu_1,\,\,x(T) = \nu_2.
	 \end{cases}
 \end{equation}
 where $\nu_1, \nu_2 \in \R$, $\Phi:\R \to \R$ is a strictly increasing homeomorphism,
 $f$ is a Carath\'eodory function
 and $a:I\times\R\to \R$ is a continuous non-negative function satisfying
 the following estimate from below:
 \begin{equation} \label{est.a.intro}
  a(t,x)\geq h(t) \ \text{for every $t\in I$ and every $x\in\R$},
 \end{equation}
 where
 $h\in C(I,\R)$ is non-negative and such that
 $1/h \in L^p(I)$ for some $p> 1$.
 Notice that, differently to other papers quoted above, here we do not
  require the positivity of the function $a$, and thus
  the equation in \eqref{eq.mainBVP-intro}
  may be singular. Consequently, as
  in \cite{CaMaPa}, we look for solutions in $W^{1,1}(I)$.

 For example, estimate \eqref{est.a.intro} is verified when $a(t,x)$ has a simpler 
 struc\-tu\-re of a product or of a sum, as
 in the following special cases:
 \begin{itemize}
 \item[$(\star)$] if $a(t,x)=h(t)\cdot b(x),$
	where $h$ is continuous, non-negative and such that 
 	$1/h\in L^p(I)$, and
 	$b$ is continuous and such that $\inf_\R b > 0$; \medskip
 \item[$(\star)$] if $a(t,x)=h(t)+b(x),$
	where $h$ is continuous, non-negative and such that 
 	$1/h\in L^p(I)$, and
 	$b$ is continuous and non-negative.
\end{itemize} 
 
 \medskip
 Our main result, Theorem \ref{thm.mainconcrete} below, yields the existence of a 
 solution of the Dirichlet problem
 \eqref{eq.mainBVP-intro} assuming a weak Wintner-Nagumo condition, similar to the one in \eqref{ip:psi2}.
  Theorem \ref{thm.mainconcrete} extends in a natural way the main result in \cite{CaMaPa}, in the case
  when $a(t,x)=k(t)$ does not depend on $x$. 
 The proof is obtained by the method of lower/upper solutions, 
 combined with a fixed point technique applied to an auxiliary functional Dirichlet problem
 (see Section \ref{sec:abstractset}). 
 In Section  \ref{sec:concrete} we provide some illustrating examples in which 
 our main result can be applied. 
 
 Finally,
 in Section \ref{sec:general} we consider different BVPs, such as the periodic problem, 
 Neu\-mann-type problem and Sturm-Liouville-type problem,
 and wi\-th clas\-sical tech\-ni\-ques we derive existence results.


\section{The abstract setting} \label{sec:abstractset}
 Let $T > 0$ be fixed and let $I := [0,T]$. Moreover, let $\nu_1,\nu_2\in\R$.
 Throughout this section, we shall be concerned with
 {BVPs} of the type
 \begin{equation} \label{eq.mainBVPAbstract}
   \begin{cases}
   \dsy\Big(\Phi\big(A_x(t)\,x'(t)\big)\Big)' = F_x(t), & \text{a.e.\,on $I$}, \\[0.2cm]
   x(0) = \nu_1,\,\,x(T) = \nu_2,
   \end{cases}
  \end{equation}
  where $\Phi, A$ and $F$ satisfy
  the following structural assumptions:
  \begin{itemize}
   \item[(H1)] $\Phi:\R\to\R$ is a \emph{strictly increasing homeomorphism}; \medskip
   \item[(H2)] $A:W^{1,p}(I)\subseteq C(I,\R)\to C(I,\R)$ is continuous
   wrt the \emph{uniform topology} of $C(I,\R)$; 
   moreover, there exist $h_1,h_2\in C(I,\R)$ such that \medskip
	\begin{itemize}
	 \item[$(\mathrm{H2})_1$] $h_1,h_2\geq 0$ on $I$ and 
	 there exists $p > 1$ such that
	 $$1/h_1,1/h_2\in L^p(I);$$
	 \item[$(\mathrm{H2})_2$] $h_1(t)\leq A_x(t)\leq h_2(t)$ for every $x\in W^{1,p}(I)$ and every $t\in I$;
	\end{itemize}
	\medskip
   \item[(H3)] $F:W^{1,p}(I)\to L^1(I)$ is continuous 
   (with respect to the usual norms) and there exists a non-negative function
   $\psi\in L^1(I)$ such that
   \begin{equation} \label{eq.psiboundsF}
    |F_x(t)| \leq \psi(t) \quad \text{for every $x\in W^{1,p}(I)$ and a.e.\,$t\in I$}.
   \end{equation}
  \end{itemize}
  \begin{rem} \label{rem.Axprop}
   We point out that, as a consequence of assumptions
   $\mathrm{(H2)}_1$ and $\mathrm{(H2)}_2$, 
   for every $x\in W^{1,p}(I)$ we have
   $A_x\geq h_1 \geq 0$ and
   $$\int_0^T\frac{1}{h_2(t)}\,\d t \leq \int_0^T\frac{1}{A_x(t)}\,\d t
   \leq \int_0^T\frac{1}{h_1(t)}\,\d t.$$
   Since $1/h_1,1/h_2\in L^p(I)$, then the same is true of 
   $1/A_x$ (for any $x\in W^{1,p}(I)$).
  \end{rem}
   In the sequel, we shall indicate by $\mathcal{F}$ the integral operator associated
   with $F$, that is, the operator $\mathcal{F}: W^{1,p}(I)\to C(I,\R)$ defined by
   $$\mathcal{F}_x(t) := \int_0^t F_x(s)\,\d s.$$
   \begin{rem} \label{rem.intFcont}
    We observe, for future reference, that $\mathcal{F}$ is \emph{continuous}
    from $W^{1,p}(I)$ to $C(I,\R)$: this follows
    from the continuity of $F$ and from the estimate
    (holding true for every $x,y\in W^{1,p}(I)$)
    \begin{equation}
     \sup_{t\in I}|\mathcal{F}_x(t)-\mathcal{F}_y(t)|
     \leq \|F_x-F_y\|_{L^1}.
    \end{equation}
    Furthermore, assumption
    \eqref{eq.psiboundsF} gives
	\begin{equation} \label{eq.estimpsiintF}
    \sup_{t\in I}|\mathcal{F}_x(t)|\leq \|\psi\|_{L^1},
    \quad \text{for every $x\in W^{1,p}(I)$}.
    \end{equation}
   \end{rem}
   \begin{defn} \label{def.solutionBVP}
  We say that a {continuous}
  function $x\in C(I,\R)$ is a \textbf{solution} of the boundary
  value problem \eqref{eq.mainBVPAbstract} if it satisfies
  the following properties:  
  \begin{itemize}
   \item[(1)] $x\in W^{1,p}(I)$ and $t\mapsto\Phi\big(A_x(t)\,x'(t)\big)\in W^{1,1}(I)$;
   \item[(2)] $\Big(\Phi\big(A_x(t)\,x'(t)\big)\Big)' = F_x(t)$ for a.e.\,$t\in I$;
   \item[(3)] $x(0) = \nu_1$ and $x(T) = \nu_2$.
  \end{itemize}
  \end{defn}
  \begin{rem} \label{rem.definitionAxcont}
   We point out that, if $x\in W^{1,p}(I)$ is a so\-lu\-tion
   of
   \eqref{eq.mainBVPAbstract}, by condition (1) in Definition \ref{def.solutionBVP}
   (and the fact that $\Phi$ is a ho\-meo\-mor\-phism, see assumption
   (H1)), there exists a \emph{unique} $\mathcal{A}_x\in C(I,\R)$ such that
   $$\mathcal{A}_x(t) = A_x(t)\,x'(t) \quad \text{for a.e.\,$t\in I$}.$$
   We shall use this fact in the next Section \ref{sec:concrete}.
  \end{rem}
  The main result of this section is the following \emph{existence result}.
  \begin{thm} \label{thm.mainabstract}
   Under the structural assumptions \emph{(H1), (H2)}
   and \emph{(H3)}, the boun\-da\-ry value problem
   \eqref{eq.mainBVPAbstract} admits at least one solution $x\in W^{1,p}(I)$.
  \end{thm}
  The proof of Theorem \ref{thm.mainabstract} requires
  some preliminary facts.
  \begin{lem} \label{lem.constant}
   For every $x\in W^{1,p}(I)$, there exists a 
   \emph{unique} $\xi_x\in\R$ such that
   \begin{equation} \label{eq.defixix}
    \int_0^T\frac{1}{A_x(t)}\,\Phi^{-1}\big(\xi_x + \mathcal{F}_x(t)\big)\,\d t 
    = \nu_2-\nu_1.
   \end{equation}
   Furthermore, there exists a universal constant $\mathbf{c}_0 > 0$ such that
   \begin{equation} \label{eq.estimxix}
    |\xi_x| \leq \mathbf{c}_0 \quad \text{for every $x\in W^{1,p}(I)$}.
   \end{equation}
  \end{lem}
  \begin{proof}
   Let $x\in W^{1,p}(I)$ be fixed and let
   $$f_x:\R\longto\R, \qquad
   f_x(\xi) := \int_0^T\frac{1}{A_x(t)}\,\Phi^{-1}\big(\xi + \mathcal{F}_x(t)\big)\,\d t .$$   
   Since
   $\mathcal{F}_x$ is continuous on $I$
   (see Remark \ref{rem.intFcont}) and since,
   by assumptions, $\Phi$ is con\-ti\-nuo\-us on the whole of $\R$, an application
   of Lebesgue's Dominated Convergence Theorem shows that $f_x\in C(\R,\R)$
   (see also Remark \ref{rem.Axprop});
   mo\-re\-o\-ver, since $\Phi$ is increasing, the same is true of $f_x$
   and, by \eqref{eq.estimpsiintF}, we have
   \begin{equation} \label{eq.touseestimxix}
    \begin{split}
    & \Phi^{-1}(\xi-\|\psi\|_{L^1})\cdot\bigg(\int_0^T\frac{1}{A_x(t)}\,\d t\bigg)
  \leq f_x(\xi) \leq \\
	& \qquad\quad 
	\leq \Phi^{-1}(\xi+\|\psi\|_{L^1})\cdot\bigg(\int_0^T\frac{1}{A_x(t)}\,\d t\bigg).
   \end{split}
   \end{equation}
   From this, we deduce that $f_x(\xi)\to\pm\infty$ as $\xi\to\pm\infty$; thus,
   by Bolzano's Theorem (and the monotonicity
   of $f_x$), there exists a unique $\xi_x\in\R$ s.t.
   $$f_x(\xi_x) = 
	\int_0^T\frac{1}{A_x(t)}\,\Phi^{-1}\big(\xi_x + \mathcal{F}_x(t)\big)\,\d t
	= 
	\nu_2-\nu_1.$$
	We now turn to prove estimate \eqref{eq.estimxix}. To this end we observe that,
	by \eqref{eq.defixix}
	and the Mean Value Theorem, there exists $t^* = t^*_x\in I$ such that
	$$\Phi^{-1}(\xi_x + \mathcal{F}_x(t^*))\cdot\bigg(\int_0^T\frac{1}{A_x(t)}\,\d t\bigg)
	= \nu_2-\nu_1;$$
	as a consequence, we obtain
	$$\xi_x + \mathcal{F}_x(t^*) = \Phi\bigg((\nu_2-\nu_1)\cdot
	\bigg(\int_0^T\frac{1}{A_x(t)}\,\d t\bigg)^{-1}\bigg).$$
    Now, by crucially exploiting Remark \ref{rem.Axprop}, we see that
	$$\bigg|(\nu_2-\nu_1)\cdot
	\bigg(\int_0^T\frac{1}{A_x(t)}\,\d t\bigg)^{-1}\bigg|
	\leq |\nu_2-\nu_1|\cdot \bigg(\int_0^T\frac{1}{h_2(t)}\,\d t\bigg)^{-1} =: \rho,$$
	\emph{for every $x\in W^{1,p}(I)$}; setting
	$M := \sup_{[-\rho,\rho]}|\Phi|$, we get (see also \eqref{eq.estimpsiintF})
	\begin{equation*}
	\begin{split}
	 |\xi_x| & \leq |\xi_x+\mathcal{F}_x(t^*)|+|\mathcal{F}_x(t^*)| \\[0.2cm]
	 & \leq \bigg|\Phi\bigg((\nu_2-\nu_1)\cdot
	\bigg(\int_0^T\frac{1}{A_x(t)}\,\d t\bigg)^{-1}\bigg)\bigg| + 
	\sup_{t\in I}|\mathcal{F}_x(t)| \\[0.2cm]
	 & \leq M + \|\psi\|_{L^1} =: \mathbf{c}_0.
	\end{split}
	\end{equation*}
	Since $\mathbf{c}_0 > 0$ does not depend on $x$, this gives
	the desired \eqref{eq.estimxix}.  
   \end{proof}
   We now consider the operator $\mathcal{P}:W^{1,p}(I)\to W^{1,p}(I)$ defined by
   \begin{equation} \label{defi.operatorT}
    \mathcal{P}_{x}(t) := \nu_1 + \int_0^t \frac{1}{A_x(s)}\,\Phi^{-1}\big(\xi_x+\mathcal{F}_x(s)\big)\,\d s,
   \end{equation}
   where $\xi_x$ is as in Lemma \ref{lem.constant}.
   We note that $\mathcal{P}$ is well-defined, in the sense
   that $\mathcal{P}_{x}\in W^{1,p}(I)$ for every $x\in W^{1,p}(I)$: 
   indeed, assumption $(\mathrm{H2})_2$ and \eqref{eq.estimpsiintF}
   give
   $$\bigg|\frac{1}{A_x(s)}\,\Phi^{-1}\big(\xi_x+\mathcal{F}_x(s)\big)\bigg|
   \leq \frac{1}{h_1(t)}\,\Phi^{-1}\big(\xi_x+\|\psi\|_{L^1}\big);$$
   thus, since $1/h_1\in L^p(I)$, we conclude that $\mathcal{P}_{x}\in W^{1,p}(I)$, as claimed.
   Fur\-ther\-mo\-re, it is not difficult to see that
    the solutions of \eqref{eq.mainBVPAbstract}
    (ac\-cor\-ding to Definition \ref{def.solutionBVP}) are precisely
    \textbf{the fixed points (in $W^{1,p}(I)$) of $\mathcal{P}$}.
    \medskip
    
    In view of this fact, we can prove Theorem \ref{thm.mainabstract} by showing that
    $\mathcal{P}$ possesses at least one fixed point in $W^{1,p}(I)$; in its turn, the existence
    of a fixed point of $\mathcal{P}$ follows from Schauder's Fixed Point Theorem if
    we are able to demonstrate that $\mathcal{P}$ enjoys the following properties: \medskip
    
     $\bullet$\,\,$\mathcal{P}$ is bounded in $W^{1,p}(I)$; \medskip
     
     $\bullet$\,\,$\mathcal{P}$ is continuous from $W^{1,p}(I)$ into itself; \medskip
     
     $\bullet$\,\,$\mathcal{P}$ is \emph{compact}. \medskip
     
     These facts are proved in the next lemmas.
    \begin{lem} \label{lem.Tbounded}
     The operator $\mathcal{P}$ defined in \eqref{defi.operatorT} is \textbf{bounded} in $W^{1,p}(I)$,
     that is, there exists a universal constant $\mathbf{c}_1 > 0$ such that
     $$\|\mathcal{P}_{x}\|_{W^{1,p}} \leq \mathbf{c}_1 \quad
     \text{for every $x\in W^{1,p}(I)$}.$$
    \end{lem}
    \begin{proof}
     For every $x\in W^{1,p}(I)$, by combining \eqref{eq.estimpsiintF} and
     \eqref{eq.estimxix}, we have 
     $$|\xi_x + \mathcal{F}_x(s)| \leq \mathbf{c}_0 + \|\psi\|_{L^1} =:\eta,
     \quad \text{for every $s\in I$};$$
     thus, if we set $\widehat{M} = \max_{[-\eta,\eta]}|\Phi^{-1}|$, we obtain
     (see assumption $(\mathrm{H2})_2$)
     \begin{equation} \label{eq.estimTbound}
      \bigg|\frac{1}{A_x(s)}\,\Phi^{-1}\big(\xi_x+\mathcal{F}_x(s)\big)\bigg|
 	 \leq \frac{M}{h_1(s)} \quad \text{for every $s\in I$},     
 	 \end{equation}
 	 and the estimate holds \emph{for every $x\in W^{1,p}(I)$}.
	With such an estimate at hand, we can easily prove the boundedness of $\mathcal{P}$:
	indeed, by \eqref{eq.estimTbound} we have
	\begin{align*}
	 \|\mathcal{P}_{x}'\|_{L^p} & =
	 \bigg(\int_0^T\bigg|\frac{1}{A_x(s)}\,\Phi^{-1}\big(\xi_x+\mathcal{F}_x(s)\big)\bigg|^p\,\d s
	 \bigg)^{1/p} \leq M\,\|1/h_1\|_{L^p}
	 \end{align*}
	 for every $x\in W^{1,p}(I)$; moreover, one has
	 \begin{align*}
	  \|\mathcal{P}_{x}\|_{L^p} & \leq
	 \bigg\{\int_0^T\bigg(|\nu_1|
	 +\int_0^t\bigg|\frac{1}{A_x(s)}\,\Phi^{-1}\big(\xi_x+\mathcal{F}_x(s)\big)
	 \bigg|\,\d s\bigg)^p\,\d t\bigg\}^{1/p} \\[0.25cm]
	 & \leq T^{1/p}\,\big(|\nu_1|+ M\,\|1/h_1\|_{L^1}\big),
	\end{align*}
	and again the estimate holds for every $x\in W^{1,p}(I)$. Summing up, if we in\-tro\-duce
	the constant (which does not depend on  $x$)
	$$\mathbf{c}_1 := 	T^{1/p}\,(|\nu_1|+ M\,\|1/h_1\|_{L^1})
	+ M\,\|1/h_1\|_{L^p} > 0,$$
	we conclude that, for every $x\in W^{1,p}(I)$, one has
	\begin{align*}
	& \|\mathcal{P}_{x}\|_{W^{1,p}}  = \|\mathcal{P}_{x}\|_{L^p}
	+ \|\mathcal{P}_{x}'\|_{L^p} \leq \mathbf{c}_1.
	\end{align*}
	This ends the proof.  
    \end{proof}
    \begin{rem} \label{rem.estimlebesgue}
     It is contained in the proof of Lemma \ref{lem.Tbounded} the following fact, which we shall repeatedly
     use in the sequel: there exists a constant $M > 0$ such that,
     for every $x\in W^{1,p}(I)$,
     \begin{equation} \label{eq.estimLebesguemain}
      \max_{t\in I}\Big|\Phi^{-1}\big(\xi_x+\mathcal{F}_x(t)\big)\Big| \leq M.
     \end{equation}
     We also highlight that, since the injection $W^{1,p}(I)\subseteq C(I,\R)$ is continuous, the boundedness
     of $\mathcal{P}$ in $W^{1,p}(I)$ implies the boundedness of 
     $\mathcal{P}$ in $C(I,\R)$: more precisely,
     there exists a real $\mathbf{c}'_1 > 0$ such that
     \begin{equation} \label{eq.TboundedCI}
      \sup_{t\in I}|\mathcal{P}_{x}(t)| \leq \mathbf{c}'_1, \quad
      \text{for every $x\in W^{1,p}(I)$}.
     \end{equation}
    \end{rem}
    We now turn to prove the continuity of $\mathcal{P}$.
    \begin{lem} \label{lem.Tcontinuous}
     The operator $\mathcal{P}$ defined in \eqref{defi.operatorT} is \textbf{continuous}
     on $W^{1,p}(I)$.
    \end{lem}
    \begin{proof}
     Let $x_0\in W^{1,p}(I)$ be fixed and let $\{x_n\}_{n\in\N}\subseteq W^{1,p}(I)$ be a sequence
     converging to $x_0$ as $n\to\infty$. We need to prove that $\mathcal{P}_{x_n}
     \to \mathcal{P}_{x_0}$ as $n\to\infty$.
     
     To this end, we arbitrarily choose a sub-sequence $\{x_{n_k}\}_{k\in\N}$ of $\{x_n\}_{n\in\N}$
     and we show that there exists a further sub-sequence $\{x_{n_{k_j}}\}_{j\in\N}$ such that
     $$\lim_{j\to\infty}\mathcal{P}_{x_{n_{k_j}}} = \mathcal{P}_{x_0} \quad
     \text{in $W^{1,p}(I)$}.$$
     First of all, by \eqref{eq.estimxix}, the sequence
     $\{\xi_{x_{n_k}}\}_{k\in\N}$ is bounded in $\R$; thus, there exist
     an increasing sequence $\{k_j\}_{j\in\N}\subseteq\N$ and a real $\xi_0\in\R$ such that
     $$\xi_j := \xi_{x_{n_{k_j}}} \to \xi_0 \quad \text{as $j\to\infty$}.$$
     Moreover, since $\mathcal{F}$ is continuous from $W^{1,p}(I)$ to $C(I,\R)$
     (see Remark \ref{rem.intFcont}), we have
     $\mathcal{F}_j := \mathcal{F}_{x_{n_{k_j}}}\to \mathcal{F}_{x_0}$ uniformly on $I$ as $j\to\infty$.
     Finally, since $A$ is continuous \emph{wrt the uniform topology}
     of $C(I,\R)$ (by (H2)), one has 
     $$\text{$A_j := A_{x_{n_{k_j}}}\to A_{x_0}$ uniformly on $I$ as $j\to\infty$}.$$
     Gathering together all these facts (and reminding that $\Phi\in C(\R,\R)$),
     we get
     \begin{equation} \label{touseLebesguederT}
	  \begin{split}      
      & \lim_{j\to\infty}\frac{1}{A_j(t)}\,
     \Phi^{-1}\big(\xi_j+\mathcal{F}_j(t)\big) \\
    &\qquad = \frac{1}{A_{x_0}(t)}\,\Phi^{-1}\big(\xi_0+\mathcal{F}_{x_0}(t)\big)
     \quad \text{for a.e.\,$t\in I$}.
     \end{split}
     \end{equation}
    From this, owing to estimate \eqref{eq.estimLebesguemain} 
    and Remark \ref{rem.Axprop}, we infer that
    \begin{equation}  \label{eq.intjconvint_one}
     \begin{split}
    & \lim_{j\to\infty}\int_0^t
    \frac{1}{A_j(s)}\,
     \Phi^{-1}\big(\xi_j+\mathcal{F}_j(s)\big)\,\d s 
     \\
     & \qquad\quad= \int_0^t
     \frac{1}{A_{x_0}(s)}\,\Phi^{-1}\big(\xi_0+\mathcal{F}_{x_0}(s)\big)\,\d s
     \quad \text{for every $t\in I$}.
     \end{split}
    \end{equation}
    In particular, since we know from Lemma \ref{lem.constant} that
    $$\int_0^T\frac{1}{A_j(s)}\,
     \Phi^{-1}\big(\xi_j+\mathcal{F}_j(s)\big)\,\d s = \nu_2-\nu_1 \quad
     \text{for every $j\in\N$},$$
     identity \eqref{eq.intjconvint_one} implies that
     $$\int_0^T\frac{1}{A_{x_0}(s)}\,\Phi^{-1}\big(\xi_0+\mathcal{F}_{x_0}(s)\big)\,\d s
     = \nu_2-\nu_1;$$
     thus, by the uniqueness property of $\xi_x$ in Lemma \ref{lem.constant}, we get $\xi_0 = \xi_{x_0}$.
     As a consequence, by exploiting the very definition of $\mathcal{P}$
     (see \eqref{defi.operatorT}), identity \eqref{eq.intjconvint_one} allows us
     to conclude that $\mathcal{P}_{x_{n_{k_j}}}\to 
     \mathcal{P}_{x_0}$ point-wise on $I$ as $j\to\infty$. \medskip
     
     To complete the proof of the lemma, we need to 
     show that the sequence
     $\mathcal{P}_{x_{n_{k_j}}}$
     actually converges to $\mathcal{P}_{x_0}$ in $W^{1,p}(I)$ as $j\to\infty$.
     To this end we first observe that, by ex\-ploi\-ting estimate \eqref{eq.estimLebesguemain}, 
     for almost every $t\in I$ one has	
     \begin{equation*} 
     \begin{split}
      & \bigg|
      \frac{1}{A_j(t)}\,
     \Phi^{-1}\big(\xi_j+\mathcal{F}_j(t)\big)
     - \frac{1}{A_{x_0}(t)}\,\Phi^{-1}\big(\xi_{x_0}+\mathcal{F}_{x_0}(t)\big)\bigg|^p 
     \leq 2^p\,\frac{M^p}{h_1^p(t)};
     \end{split}
     \end{equation*}
 	 as a consequence, since
     $1/h_1\in L^p(I)$ (by assumption (H2)), a standard ap\-pli\-ca\-tion of
     Lebesgue's Dominated Convergence Theorem gives (see also \eqref{touseLebesguederT})
     \begin{equation*}
      \begin{split}
       & \lim_{j\to\infty}\|\mathcal{P}_{x_{n_{k_j}}}'-\mathcal{P}_{x_0}'\|_{L^p}^p \\
     & \quad = \lim_{j\to\infty}\int_0^T
     \bigg|
      \frac{1}{A_j(t)}\,
     \Phi^{-1}\big(\xi_j+\mathcal{F}_j(t)\big)
     - \frac{1}{A_{x_0}(t)}\,\Phi^{-1}\big(\xi_{x_0}+\mathcal{F}_{x_0}(t)\big)\bigg|^p\d t  = 0.
     \end{split}
     \end{equation*}
     On the other hand, since $\mathcal{P}$ is bounded in $C(I,\R)$ 
     (see Remark \ref{rem.estimlebesgue}), one has
     $$\text{$|\mathcal{P}_{x_{n_{k_j}}}(t)-
     \mathcal{P}_{x_0}(t)|^p \leq 2^{p}\,\mathbf{c}'_1$ for every $t\in I$};$$
     thus, again by
     Lebesgue's Dominated Convergence Theorem, we get
     \begin{equation*}
     \begin{split}
     & \lim_{j\to\infty}\|\mathcal{P}_{x_{n_{k_j}}}-\mathcal{P}_{x_0}\|^p_{L^p} \\
     & \qquad = \lim_{j\to\infty}\int_0^T|\mathcal{P}_{x_{n_{k_j}}}(t)-
     \mathcal{P}_{x_0}(t)|^p\,\d t  = 0.
     \end{split}
     \end{equation*}
     Gathering together these facts, we conclude that 
     $\|\mathcal{P}_{x_{n_{k_j}}}-
     \mathcal{P}_{x_0}\|_{W^{1,p}}\to 0$ as $j\to\infty$, and this finally
     completes the demonstration of the lemma.  
    \end{proof}
    Finally, we prove that $\mathcal{P}$ is compact.
    \begin{lem} \label{lem.Tcompact}
    The operator $\mathcal{P}$ defined in \eqref{defi.operatorT} is \textbf{compact}
     on $W^{1,p}(I)$.
    \end{lem}
    \begin{proof}
     Let $\{x_n\}_{n\in\N}\subseteq W^{1,p}(I)$ be bounded. We need to prove that the sequence
     $\{\mathcal{P}_{x_n}\}_{n\in\N}$ possesses a sub-sequence which is convergent
     (in the $W^{1,p}$-norm)
     to some function $y_0\in W^{1,p}(I)$.
     
	Fist of all, since $\{\xi_{x_n}\}_{n\in\R}$ is bounded in $\R$
	(see \eqref{eq.estimxix}), there exist
	a real $\xi_0$ and a sub-sequence of $\{x_n\}_{n\in\N}$, denoted again
	by $\{x_n\}_{n\in\N}$, such that
	\begin{equation} \label{eq.limxin}
	 \lim_{n\to\infty}\xi_{x_n} = \xi_0 \qquad \text{and}
	\qquad |\xi_0|\leq \mathbf{c}_0.
	\end{equation}
	Moreover, since $\{x_n\}_{n\in\N}$ is \emph{bounded}	
	in $W^{1,p}(I)$ and $p > 1$, there exist a suitable function
	$x_0\in W^{1,p}(I)$ and another sub-sequence of $\{x_n\}_{n\in\N}$,
	which we still denote by $\{x_n\}_{n \in \N}$, such that
	$x_n \to x_0$ uniformly on $I$ as $n\to\infty$.
	
	As a consequence, since the operator $A$ is continuous with respect to
	the uniform
	topology of $C(I,\R)$ (by assumption (H2)),
	we have 
	\begin{equation} \label{eq.limAxn}
	\text{$A_{x_n}\to A_{x_0}$ uniformly on $I$ as $j\to\infty$}.
	\end{equation}
	We now observe that, by assumption (H3), we have the estimate
	$$\text{$F_{x_n}(t)\leq \psi(t)$, holding true for a.e.\,$t\in I$ and every
	$n\in\N$};$$
	thus, $\{F_{x_n}\}_{n\in\N}$ is \emph{bounded and equi-integrable}
	in $L^1$. Owing to the Dunford-Pettis Theorem, we infer the existence
	of a function $g\in L^1(I)$ and of 
	another sub-sequence of $\{x_n\}_{n\in\N}$, denoted once again
	by $\{x_n\}_{n\in\N}$, such that \medskip
	
	$(\star)$\,\,$\dsy \lim_{n\to\infty}\int_0^T F_{x_n}(s)\,v(s)\,d s = \int_0^T g(s)\,v(s)\,\d s
	\quad \text{for every $v\in L^\infty(I)$}$; \medskip
	
	$(\star)$\,\,$\|g\|_{L^1} \leq \|\psi\|_{L^1}.$ \medskip
	
	\noindent	Choosing $v$ as the indicator function
	of $[0,t]$ (with $t\in I$), we get
	\begin{equation} \label{eq.limitFxn}
	\begin{split}
	 & \mathcal{F}_{x_n}(t) \to \mathcal{G}(t) := \int_0^t g(s)\,\d s \quad
	 \text{for every $t\in I$} \\[0.2cm]
	 & \qquad\qquad \text{and} \quad \sup_{t\in I}|\mathcal{G}(t)|\leq \|\psi\|_{L^1}.
	 \end{split}
	\end{equation}
	Gathering together \eqref{eq.limxin}, \eqref{eq.limAxn} and
	\eqref{eq.limitFxn}, we deduce that
	\begin{equation} \label{eq.limitderivatcompact}
	  \begin{split}      
      & \lim_{n\to\infty}\frac{1}{A_{x_n}(t)}\,
     \Phi^{-1}\big(\xi_{x_n} + \mathcal{F}_{x_n}(t)\big) \\
    &\qquad = \frac{1}{A_{x_0}(t)}\,\Phi^{-1}\big(\xi_0+\mathcal{G}(t)\big)
     \quad \text{for a.e.\,$t\in I$}.
     \end{split}
	\end{equation}
	From this, owing to
	\eqref{eq.limxin}, \eqref{eq.limitFxn} 
	and Remark \ref{rem.Axprop}, we conclude that
	\begin{equation} \label{eq.estimdertouse}
	 \bigg|\frac{1}{A_{x_0}(t)}\,\Phi^{-1}\big(\xi_0+\mathcal{G}(t)\big)\bigg|
	\leq \frac{M}{h_1(t)}\in L^p(I) \quad \text{for a.e.\,$t\in I$}
	\end{equation}
	and that, for every $t\in I$, one has
	\begin{align*}
	 & \lim_{n\to\infty}\mathcal{P}_{x_n}(t) = \lim_{n\to\infty}
	 \Big\{\nu_1+\int_0^t
	 \frac{1}{A(x_n)(s)}\,
     \Phi^{-1}\big(\xi_{x_n}+\mathcal{F}_{x_n}(s)\big)\Big\} \\[0.2cm]
     & \quad = \nu_1 + \int_0^t
     \frac{1}{A_{x_0}(s)}\,\Phi^{-1}\big(\xi_0+\mathcal{G}(s)\big) =:y_0(t) \quad
     \text{for every $t\in I$}.
	\end{align*}
	To complete the proof of the lemma, we need to show that the sequence
	$\{\mathcal{P}_{x_n}\}_{n}$ actually converges to 
	$y_0$ \emph{in} $W^{1,p}(I)$ as $n\to\infty$.
	
	On the one hand, by using estimate \eqref{eq.estimdertouse}
	and by
	arguing exactly as in the proof
	of Lemma \ref{lem.Tcontinuous}, we easily recognize that
	     \begin{equation*}
      \begin{split}
       & \lim_{n\to\infty}\|\mathcal{P}_{x_n}'-y_0'\|_{L^p}^p = \\[0.2cm]
     & \,\, = \lim_{n\to\infty}\int_0^T
     \bigg|
      \frac{1}{A_{x_n}(t)}\,
     \Phi^{-1}\big(\xi_{x_n}+\mathcal{F}_{x_n}(t)\big)
     - \frac{1}{A_{x_0}(t)}\,\Phi^{-1}\big(\xi_0+\mathcal{G}(t)\big)\bigg|^p\d t  = 0.
     \end{split}
     \end{equation*}
     On the other hand, since $\mathcal{P}_{x_n}\to y_0$ point-wise on $I$, from
     \eqref{eq.TboundedCI} we get 
     $$\text{$|y_0(t)|\leq \mathbf{c}'_1$ for every $t\in I$};$$
     hence, by arguing once again as in the proof of Lemma \ref{lem.Tcontinuous}, we conclude that
     \begin{equation*}
     \begin{split}
     & \lim_{n\to\infty}\|\mathcal{P}_{x_n}-y_0\|^p_{L^p} \\
     & \qquad = \lim_{n\to\infty}\int_0^T|\mathcal{P}_{x_n}(t)-y_0(t)|^p\,\d t  = 0.
     \end{split}
     \end{equation*}
	Summing up, $\mathcal{P}_{x_n}\to y_0$ in $W^{1,p}(I)$ as $n\to\infty$,
	and the proof is complete. 
	\end{proof}
	Gathering Lemmas \ref{lem.Tbounded}, \ref{lem.Tcontinuous} and
	\ref{lem.Tcompact}, we can prove Theorem \ref{thm.mainabstract}.
	\begin{proof} [of Theorem \ref{thm.mainabstract}]
	 We have already pointed out that a function $x\in W^{1,p}(I)$ is a solution
	 of the boundary value problem \eqref{eq.mainBVPAbstract}
	 \emph{if and only if} $x$ is a fixed point of the operator $\mathcal{P}$
	 defined in \eqref{defi.operatorT}. On the other hand, since $\mathcal{P}$ 
	 is bounded, continuous
	 and compact on the Banach space $W^{1,p}(I)$, the Schauder Fixed Point Theorem ensures
	 the existence of (at least) one $x\in W^{1,p}(I)$ such that 
	 $\mathcal{P}_{x} = x$, and thus
	 the problem \eqref{eq.mainBVPAbstract} possesses at least
	 one solution. 
	\end{proof}
\section{The Dirichlet problem for singular ODEs} \label{sec:concrete}
	In this section, we exploit the existence result
	in Theorem \ref{thm.mainabstract} in order to prove the solvability
	of boundary value problems of the following type
	\begin{equation} \label{eq.mainBVPconcrete}
	 \begin{cases}
   		\dsy\Big(\Phi\big(a(t,x(t))\,x'(t)\big)\Big)' = f(t,x(t),x'(t)), & \text{a.e.\,on $I$}, \\[0.2cm]
   		x(0) = \nu_1,\,\,x(T) = \nu_2.
	 \end{cases}
	\end{equation}
	As in Section \ref{sec:abstractset}, $I = [0,T]$ (for some real $T > 0$)
	and $\nu_1,\nu_2\in \R$; furthermore, the functions $\Phi, a$ and $f$ satisfy
	the following \emph{structural assumptions}:
	\begin{itemize}
	 \item[(A1)] $\Phi:\R\to\R$ is a \emph{strictly increasing homeomorphism}; \medskip
	 \item[(A2)] $a\in C(I\times\R,\R)$ and
	 there exists $h\in C(I,\R)$ such that \medskip
	 \begin{itemize}
	  \item[$(\mathrm{A2})_1$] $h \geq 0$ on $I$ and there exists $p > 1$
	  such that 
	  $$1/h\in L^p(I);$$
	  \item[$(\mathrm{A2})_2$] $a(t,x) \geq h(t)$ for every
	  $t\in I$ and every $x\in\R$; \medskip
	 \end{itemize}
	 \item[(A3)] $f:I\times\R^2\to\R$ is a \emph{Carath\'{e}odory function},
	 that is, \medskip
	 
	 \noindent $(*)$\,\,the map $t\mapsto f(t,x,y)$ is measurable on $I$, for every
	 $(x,y)\in\R^2$; \vspace*{0.1cm}
	 
	 \noindent $(*)$\,\,the map $(x,y)\mapsto f(t,x,y)$ is continuous
	 on $\R^2$, for a.e.\,$t\in I$.
	\end{itemize}
	\begin{rem} \label{rem.propatxused}
	 As in Section \ref{sec:abstractset} we point out that, as a consequence of
	 $(\mathrm{A2})_1$-$(\mathrm{A2})_2$, 
	 for every $(t,x)\in I\times\R$ one has $a(t,x)\geq h(t)\geq 0$ and
	 $$0 \leq \int_0^T\frac{1}{a(t,x(t))}\,\d t
	 \leq \int_0^T \frac{1}{h(t)}\,\d t,$$
	 for any measurable function $x:I\to\R$.
	 Hence, $t\mapsto a(t,x(t))\in L^p(I)$.
	\end{rem}	
	
	We now give the definition of \emph{solution} of the problem \eqref{eq.mainBVPconcrete}.
	\begin{defn} \label{def.solconcrete}
	We say that a continuous function $x\in C(I,\R)$ is a \textbf{solution}
	of the Dirichlet problem \eqref{eq.mainBVPconcrete} if it satisfies the following
	properties:
	\begin{itemize}
	\item[(1)] $x\in W^{1,1}(I)$ and $t\mapsto \Phi\big(a(t,x(t))\,x'(t)\big)\in W^{1,1}(I)$;
   	\item[(2)] $\Big(\Phi\big(a(t,x(t))\,x'(t)\big)\Big)' = f(t,x(t),x'(t))$ for almost every $t\in I$;
   	\item[(3)] $x(0) = \nu_1$ and $x(T) = \nu_2$.
  \end{itemize}
  If $x$ fulfills only (1) and (2), we say that 
  $x$ is a \textbf{solution of the ODE}
  \begin{equation} \label{eq.maiODEtocite}
   \dsy\Big(\Phi\big(a(t,x(t))\,x'(t)\big)\Big)' = f(t,x(t),x'(t)).
  \end{equation}
  \end{defn}
  In order to clearly state the main result of this section,
  we also need to introduce the definition
  of \emph{upper/lower solution}
  of the equation in \eqref{eq.mainBVPconcrete}.
  \begin{defn} \label{def.lowerupper}
   We say that a continuous function $\alpha\in C(I,\R)$ is a \textbf{lower} [respectively
   \textbf{upper}] \textbf{solution}
   of the differential equation \eqref{eq.maiODEtocite} if
   \begin{itemize}
	\item[(1)] $\alpha\in W^{1,1}(I)$ and $t\mapsto \Phi\big(a(t,\alpha(t))\,\alpha'(t)\big)\in W^{1,1}(I)$;
   	\item[(2)] $\Big(\Phi\big(a(t,\alpha(t))\,\alpha'(t)\big)\Big)' \geq\,[\leq]\,\, 
   	f(t,\alpha(t),\alpha'(t))$ for almost every $t\in I$.
   	\end{itemize}
  \end{defn}
  \begin{rem} \label{rem.contax}
   If $x\in W^{1,1}(I)$ is a solution
   of the problem \eqref{eq.mainBVPconcrete}, we denote
   by
   $\mathcal{A}_x$ the unique continuous function on $I$ such that
   (see also Remark \ref{rem.definitionAxcont})
   $$\mathcal{A}_x(t) = a(t,x(t))\,x'(t) \quad \text{for a.e.\,$t\in I$}.$$
   Notice that, as a consequence of
   condition (1) in Definition \ref{def.solconcrete} (and again of the fact
   that $\Phi$ is a homeomorphism, see (H1)), such a function exists. \medskip

   Analogously, if $\alpha\in W^{1,1}(I)$ is a lower/upper solution
   of \eqref{eq.maiODEtocite}, we denote by $\mathcal{A}_\alpha$
   the unique continuous function on $I$ such that
   $$\mathcal{A}_\alpha(t) = a(t,\alpha(t))\,\alpha'(t) \quad \text{for a.e.\,$t\in I$}.$$	
   The existence of such a function follows from (1) in Definition \ref{def.lowerupper}.
  \end{rem}  
  We are ready to state our main existence result.
  \begin{thm} \label{thm.mainconcrete}
   Let us assume that, together with the structural assumptions \emph{(A1)}-to-\emph{(A3)},
   the following additional hypotheses are satisfied:
   \begin{itemize}
   \item[{(A1')}] there exists a pair of lower and upper solutions $\alpha,\beta \in W^{1,1}(I)$ of
   the dif\-fe\-ren\-tial equation \eqref{eq.maiODEtocite} such that
   $\alpha(t)\leq \beta(t)$ for every $t\in I$;
   \item[{(A2')}] for every $R > 0$ and every non-negative function $\gamma\in L^p(I)$ there
   exists a non-negative function $h = h_{R,\gamma}\in L^p(I)$ such that
   \begin{equation} \label{eq.estimfh}
     |f(t,x,y(t))|\leq h_{R,\gamma}(t)
	\end{equation}    
	for a.e.\,$t\in I$, every $x\in\R$ with 
	$|x|\leq R$ and every function $y\in L^p(I)$
	satisfying $|y(t)|\leq \gamma(t)$ a.e.\,on $I$.
	\item[{(A3')}] there exist a constant $H > 0$, a non-negative function
	$\mu\in L^q(I)$ \emph{(}for some $1<q\leq\infty$\emph{)}, a non-negative function
	$l\in L^1(I)$ and a non-negative mea\-su\-rable
	function $\psi:(0,\infty)\to(0,\infty)$ such that
	\begin{align}
	 &(\star)\,\,1/\psi\in L^1_{\loc}(0,\infty)
	 \quad \text{and} \quad 
	 \int_1^\infty \frac{1}{\psi(t)}\,\d t = \infty; \label{eq.integralpsidiv} \\[0.2cm] 
	 &(\star)\,\,|f(t,x,y)| \leq \psi\Big(|\Phi(a(t,x)\,y)|\Big)
	 \cdot\Big(l(t)+\mu(t)\,|y|^{\frac{q-1}{q}}\Big); \label{eq.Nagumocondition}
	\end{align}
	for a.e.\,$t\in I$, every $x\in [\alpha(t),\beta(t)]$ and every $y\in\R$ with
	$|y| \geq H$.
   \end{itemize}
   Then, for every $\nu_1\in [\alpha(0),\beta(0)]$ and every $\nu_2\in [\alpha(T),\beta(T)]$,
   the \emph{(}singular\emph{)} 
   Dirichlet problem \eqref{eq.mainBVPconcrete} possesses a solution
   $x\in W^{1,p}(I)\subseteq W^{1,1}(I)$ 
   \emph{(}where $p > 1$ as is assumption \emph{(A2)}\emph{)}, further satisfying
	\begin{equation} \label{eq.xbetweenUL}   
   \alpha(t)\leq x(t)\leq \beta(t) \quad \text{for every $t\in I$}.
   \end{equation}
   Furthermore, if $M > 0$ is \emph{any}
   real number such that $\sup_I|\alpha|,\,\sup_I|\beta| \leq M$,
   it is possible to find a real $L_M > 0$, \emph{only depending on $M$},
   such that
   \begin{align} 
    \max_{t\in I}\big|x(t)\big|\leq M \quad \text{\emph{and}} \quad
    \max_{t\in I}\big|\mathcal{A}_x(t)\big| \leq L_M. \label{eq.estimsolML}
   \end{align}
  \end{thm}
  The main idea behind the proof of Theorem \ref{thm.mainconcrete} is to think
  of the Dirichlet problem \eqref{eq.mainBVPconcrete} as a \emph{particular case}
  of an abstract BVPs of the form \eqref{eq.mainBVPAbstract}, and then
  to apply the existence result contained in Theorem \ref{thm.mainabstract}.
  
  Unfortunately, we \emph{cannot} directly apply our Theorem \ref{thm.mainabstract}
  to the problem \eqref{eq.mainBVPconcrete}: in fact, in general,
  we cannot expect the (well-defined) functional
  $$W^{1,p}(I)\ni x \mapsto F_x := f(t,x(t),x'(t)) \in L^1(I)$$
  to satisfy assumption (H3) (or, more precisely, estimate \eqref{eq.psiboundsF}).
	  
  Thus, following an approach similar to that exploited
  by \cite{FerrPap,MaPa2017}, we introduce
  a suitable \emph{truncated version}
  of problem \eqref{eq.mainBVPconcrete},
  to which Theorem \ref{thm.mainconcrete}
  can apply. 
  To this end, to simplify the notation,
  we first fix some relevant constants we shall
  need for the proof of Theorem \ref{thm.mainconcrete}; henceforth, we suppose
  that all the assumption in the statement of
  Theorem \ref{thm.mainconcrete} are satisfied. \medskip
  
	\noindent Let $M > 0$ be any real number such that 
  $\sup_I|\alpha|,\,\sup_I|\beta|
  \leq M$ and let $H > 0$ be the constant appearing in
  as\-sum\-ption (A3'); moreover, we define
  \begin{equation} \label{eq.defa0ash2}
   a_0 := \max\big\{a(t,x)\,:\,(t,x)\in I\times [-M,M]\big\}
  \end{equation}   
  We choose a real $N > 0$ such that
  \begin{equation} \label{eq.choiceN}
  \begin{split}
  & N > \max\bigg\{H, \frac{2M}{T}\bigg\}\cdot a_0
  \qquad \text{and} \\ 
  & \qquad \Phi(N)\cdot\Phi(-N) < 0\quad (\text{with $\Phi(N) > 0$});
   \end{split}
  \end{equation}
   accordingly, we fix $L_M > 0$ in such a way that (see \eqref{eq.integralpsidiv})
   \begin{equation} \label{eq.choiceL2}
    \begin{split}
     \min\bigg\{\int_{\Phi(N)}^{\Phi(L_M)}&\frac{1}{\psi(s)}\,\d s,
	 \int_{-\Phi(-N)}^{-\Phi(-L_M)}\frac{1}{\psi(s)}\,\d s\bigg\} 
	  > \|l\|_{L^1} + 
	 \|\mu\|_{L^q}\,(2M)^{\frac{q-1}{q}}. 
	 \end{split}
	\end{equation}
	Notice that $L_M$ depends on $M$ (and also on $l$ and $\mu$),
	but not on $\alpha,\,\beta$ nor on $\nu_1$ and $\nu_2$.
	Introducing the functions 
	$$\gamma_0 := L_M/h \in L^p(I)\qquad\text{and}\qquad
	\hat{\gamma} := \gamma_0+|\alpha'|+|\beta'|\in L^1(I),$$ 
	we then consider the following
	\emph{truncating operators}:
	\begin{eqnarray*}
     & \mathcal{T}: W^{1,1}(I)\longto W^{1,1}(I), \qquad & \mathcal{T}(x)(t) :=
    \begin{cases}
     \alpha(t), & \text{if $x(t) < \alpha(t)$}; \\
     x(t), & \text{if $x(t)\in [\alpha(t),\beta(t)]$}; \\
     \beta(t), & \text{if $x(t) > \beta(t)$};
    \end{cases} \\[0.25cm]
     & \mathcal{D}: L^1(I)\longto L^1(I), \qquad  & \mathcal{D}(z)(t) :=
    \begin{cases}
     -\hat{\gamma}(t), & \text{if $z(t) < -\hat{\gamma}(t)$}; \\
     z(t), & \text{if $|z(t)|\leq \hat{\gamma}(t)$}; \\
     \hat{\gamma}(t), & \text{if $z(t) > \hat{\gamma}(t)$}.
    \end{cases}
   \end{eqnarray*}
   We also consider the truncated function $f^*:I\times\R^2\to\R$ defined by
   $$f^*(t,x,y) := \begin{cases}
    f\big(t,\beta(t),\beta'(t)\big) + \arctan\big(x-\beta(t)\big), & \text{if $x > \beta(t)$}; \\
    f(t,x,y), & \text{if $x\in [\alpha(t),\beta(t)]$}; \\
    f\big(t,\alpha(t),\alpha'(t)\big) + \arctan\big(x-\alpha(t)\big), & \text{if $x < \alpha(t)$}.  
   \end{cases}
   $$
   By means of the function $f^*$ and of
   the operators $\mathcal{T}$ and $\mathcal{D}$, we are finally in a position
   to introduce a ``truncated version'' of the Dirichlet problem
   \eqref{eq.mainBVPconcrete}:
      \begin{equation} \label{eq.BVPtrunc}
    \begin{cases}
    \dsy\bigg(\Phi\Big(a\big(t,\mathcal{T}(x)(t)\big)\,x'(t)\Big)\bigg)' = 
    f^*\Big(t, x(t), \mathcal{D}\big(\mathcal{T}(x)'(t)\big)\Big), & \text{a.e.\,on $I$}, \\[0.25cm]
   x(0) = \nu_1,\,\,x(T) = \nu_2.
    \end{cases}
   \end{equation}
   
   The next proposition shows that the ``abstract'' existence
   result in Theo\-rem \ref{thm.mainabstract}
   does apply to the ``truncated'' Dirichlet problem \eqref{eq.BVPtrunc}.
   \begin{prop} \label{prop.existenceBVPtrunc}
    Let the above assumptions and notation apply. Then, there exists (at least)
    one solution $x\in W^{1,p}(I)$ of the Dirichlet problem \eqref{eq.BVPtrunc}.
   \end{prop}
   \begin{proof}
    We consider the operators $A$ and $F$ defined as follows:
	\begin{align*}
   & A: W^{1,p}(I)\longto C(I,\R), \qquad A_x(t) := a\big(t,\mathcal{T}(x)(t)\big), \\
   & F: W^{1,p}(I)\longto L^1(I), \qquad F_x(t) := 
    f^*\Big(t, x(t), \mathcal{D}\big(\mathcal{T}(x)'(t)\big)\Big) 
   \end{align*}
   By means of these operators, the problem \eqref{eq.BVPtrunc} can be
   re-written as
   \begin{equation*}
    \begin{cases}
    \dsy\Big(\Phi\big(A_x(t)\,x'(t)\big)\Big)' = F_x(t), & \text{a.e.\,on $I$}, \\[0.2cm]
   x(0) = \nu_1,\,\,x(T) = \nu_2.
    \end{cases}
   \end{equation*}
   We claim that $A$ and $F$ satisfy assumptions (H2) and (H3) in Theorem \ref{thm.mainabstract}.
   
   First of all, since $\mathcal{T}$ is continuous with respect to the uniform
   topology of $C(I,\R)$ (as is very easy to see) and since, by the choice of $M$, we have
   $$-M \leq \alpha(t) \leq \mathcal{T}(x)(t) \leq \beta(t)\leq M \quad \text{for every
   $t\in I$},$$
   the uniform continuity of $a$
   on $I\times[-M,M]$ implies that $A$ is con\-ti\-nuo\-us
   from $W^{1,p}(I)$ (as a subspace of $C(I,\R)$) to $C(I,\R)$.
   Moreover, by \eqref{eq.defa0ash2} one has
   $$A_x(t) = a\big(t,\mathcal{T}(x)(t)\big)\leq a_0 \quad \text{for every $x\in W^{1,p}(I)$
	and every $t\in I$}.$$
   Finally, by assumption (A2), there exists 
   $h\in C(I,\R)$ such that \medskip
   
 	$(\star)$\,\,$h\geq 0$ and $1/h\in L^p(I)$; \medskip
 	
	$(\star)$\,\,$A_x(t)\geq h(t)$ for every $x\in W^{1,p}(I)$ and every $t\in I$. \medskip
	
	\noindent Thus, the operator $A$ satisfies (H2) in Theorem \ref{thm.mainabstract}
	(with $h_2(t) \equiv a_0$). \medskip
	
	As for the functional $F$, by arguing exactly as in \cite[Theorem 3.1]{CaMaPa}, one can recognize that
	it is continuous from $W^{1,p}(I)$ to $L^1(I)$ and that
	$$|F_x(t)| \leq \Theta(t) := h_{M,\hat{\gamma}}(t)+\frac{\pi}{2}$$
	for every $x\in W^{1,p}(I)$ and almost every $t\in I$
	(here, $h_{M,\hat{\gamma}}$ is the function appearing in assumption (A2')
	and corresponding to $M$ and $\hat{\gamma} = \gamma_0 + |\alpha'|+|\beta'|$).
	Since, obviously, $\Theta \in L^1(I)$, we con\-clu\-de that
	$F$ satisfies (H3) in Theorem \ref{thm.mainabstract}.
	
	Gathering together all these facts, we are allowed to apply
	Theorem \ref{thm.mainabstract} to problem \eqref{eq.BVPtrunc}, which therefore
	admits a solution $x\in W^{1,p}(I)$.
   \end{proof}
	We now turn to prove that any solution
	of
	\eqref{eq.BVPtrunc} actually solves \eqref{eq.mainBVPconcrete}.   
	\begin{prop} \label{prop.equivalenceBVPs}
	 Let the above assumptions and notation do apply, and let $x\in W^{1,p}(I)$ be any
	 solution of the truncated problem \eqref{eq.BVPtrunc}. 
	 
	 Then the following facts hold:
	 \begin{itemize}
	  \item[{(i)}] $\alpha(t)\leq x(t)\leq \beta(t)$ for every $t\in I$; 
	  \item[{(ii)}] $\sup_I|x|\leq M$;
	  \item[{(iii)}] $\big|\mathcal{A}_x(t)\big|\leq L_M$ for every $t\in I$;
	  \item[{(iv)}] $|x'(t)|\leq L_M/h(t) = \gamma_0(t)$ for a.e.\,$t\in I$.
	 \end{itemize}
	\end{prop}
	\begin{proof}
	 Let $x\in W^{1,p}(I)$ be any solution of 
	\eqref{eq.BVPtrunc}. According to Remark \ref{rem.definitionAxcont}, we denote
	by $\mathcal{A}_x$ the unique continuous fun\-ction on $I$ such that
	$$\mathcal{A}_x(t) = A(x(t))\,x'(t) = a\big(x,\mathcal{T}(x)(t)\big)\,x'(t) \quad
	\text{for a.e.\,$t\in I$}.$$
	Once we have proved that $x(t)\in[\alpha(t),\beta(t)]$ for all $t\in I$, we shall obtain
	$$\mathcal{A}_x(t) = a(t,x(t))\,x'(t) \quad \text{for a.e.\,$t\in I$}.$$
	Let us then turn to prove statements (i)-to-(iii). \medskip
	
	 (i)\,\,Let us assume, by contradiction, that $x(\overline{t})\notin
	  [\alpha(\overline{t}),\beta(\overline{t})]$ for some $\overline{t}\in I$; moreover,
	  to fix ideas, let us suppose that $x(\overline{t}) < \alpha(\overline{t})$.
	  
	  Since, by assumptions, $\nu_1 = x(0) \geq \alpha(0)$
	  and $\nu_2 = x(T) \geq \alpha(T)$, it is possible to find
	  suitable points $t_1,t_2,\theta \in I$, with $t_1 < \theta < t_2$, such that \medskip
	  
	  (a)\,\,$\dsy x(\theta) - \alpha(\theta) = \min_{t\in I}\big(x(t)-\alpha(t)\big) < 0;$ \medskip
	  
	  (b)\,\,$x(t_1) - \alpha(t_1) = x(t_2)-\alpha(t_2) = 0$ and $x<\alpha$
   		on $(t_1,t_2)$. \medskip
	  
	  \noindent In particular, from (b) we infer that $\mathcal{T}(x) \equiv \alpha$
	  on $(t_1,t_2)$ and that
	  \begin{equation*}
	   \begin{split}
	    f^*\Big(t, x(t), \mathcal{D}\big(\mathcal{T}(x)'(t)\big)\Big)  
	  & = f(t,\alpha(t),\alpha'(t)) + \arctan\big(x(t)-\alpha(t)\big) \\
	  &  < f(t,\alpha(t),\alpha'(t)), \qquad \text{for a.e.\,$t\in (t_1,t_2)$}.
	  \end{split}
	  \end{equation*}
	  As a consequence, since $x$ solves
	  the Dirichlet problem \eqref{eq.BVPtrunc} and $\alpha$ is a lower solution
	  of the ODE \eqref{eq.maiODEtocite}, for almost every $t\in (t_1,t_2)$ we obtain
	  \begin{equation} \label{eq.touseforabsurd}
	   \begin{split}
    & \Big(\Phi\big(\mathcal{A}_x(t)\big)\Big)' 
    = \bigg(\Phi\Big(a\big(t,\mathcal{T}(x)(t)\big)\,x'(t)\Big)\bigg)' \\[0.2cm]
    & \qquad\quad= f^*\Big(t, x(t), \mathcal{D}\big(\mathcal{T}(x)'(t)\big)\Big)  
    < f(t,\alpha(t),\alpha'(t)) \\[0.2cm]
    & \qquad\quad \leq \bigg(\Phi\Big(a\big(t,\alpha(t)\big)\,\alpha'(t)\Big)\bigg)'
    = \Big(\Phi\big(\mathcal{A}_\alpha(t)\big)\Big)'.
	   \end{split}
	  \end{equation}
	 We now introduce the subsets $I_1,I_2$ of $I$ defined as follows:
	 $$I_1 := \{t\in (t_1,\theta): x'(t) < \alpha'(t)\} \quad \text{and} \quad
	 I_2 := \{t\in (\theta,t_2): x'(t) > \alpha'(t)\}.$$ 
	 Since $x < \alpha$ on $(t_1,t_2)$, it is readily seen that
	 both $I_1$ and $I_2$ must have positive measure; thus, it is possible to find
	 $\tau_1\in I_1$ and $\tau_2\in I_2$ such that \medskip
	 
	 $(\star)$\,\,$0 < h_1(\tau_i) \leq a(\tau_i,\alpha(\tau_i))$ for $i = 1,2$; \medskip
	 
	 $(\star)$\,\,$\mathcal{A}_\alpha(\tau_i) = a(\tau_i,\alpha(\tau_i))\,\alpha'(\tau_i)$ for $i = 1,2$
	 (see Remark \ref{rem.contax}); \medskip
	 
	 $(\star)$\,\,$\mathcal{A}_x(\tau_i) = a(\tau_i,\mathcal{T}(x)(\tau_i))\,x'(\tau_i)
	 = a(\tau_i,\alpha(\tau_i))\,x'(\tau_i)$ for $i = 1,2$. \medskip 
	 
	 \noindent From this, by integrating both sides of inequality
	 \eqref{eq.touseforabsurd} on $[\tau_1,\theta]$, we get
	 \begin{equation*}
	 \begin{split}
	  & \Phi\big(\mathcal{A}_x(\theta)\big)
	  - \Phi\Big(a\big(\tau_1,\alpha(\tau_1)\big)\,x'(\tau_1)\Big) \leq 
	  \Phi\big(\mathcal{A}_\alpha(\theta)\big)
		- \Phi\Big(a\big(\tau_1,\alpha(\tau_1)\big)\,\alpha'(\tau_1)\Big);
	 \end{split}
	 \end{equation*}
	hence, by the choice of $\tau_1$ and the fact
	that $\Phi$ is strictly increasing, one has
	\begin{equation} \label{eq.incontradiction}
	\Phi\big(\mathcal{A}_x(\theta)\big) - 
	\Phi\big(\mathcal{A}_\alpha(\theta)\big) < 0.
	\end{equation}
	On the other hand, if we integrate both sides
	of \eqref{eq.touseforabsurd} on $[\theta,\tau_2]$ we get
	\begin{equation*}
	 \begin{split}
	  & \Phi\Big(a\big(\tau_2,\alpha(\tau_2)\big)\,x'(\tau_2)\Big) - 
	  \Phi\big(\mathcal{A}_x(\theta)\big) \leq 
	   \Phi\Big(a\big(\tau_2,\alpha(\tau_2)\big)\,\alpha'(\tau_2)\Big)
	   - \Phi\big(\mathcal{A}_\alpha(\theta)\big)
	 \end{split}
	 \end{equation*}
	 and thus, by the choice of $\tau_2$ and again the
	  monotonicity of $\Phi$,
	 we obtain
	 \begin{equation*}
	 \begin{split}
	  & \Phi\big(\mathcal{A}_x(\theta)\big) - 
		\Phi\big(\mathcal{A}_\alpha(\theta)\big) > 0,
	 \end{split}
	 \end{equation*}
	 This is clearly in contradiction with \eqref{eq.incontradiction}, hence
	 $x\geq \alpha$ on $I$. By arguing analogously one can prove
	 that $x\leq\beta$ on $I$, and statement (i) is established. \medskip
	 
	 \noindent (ii)\,\,By statement (i) and the choice of $M$, we immediately get
	 $$-M \leq \alpha(t)\leq x(t)\leq\beta(t)\leq M \quad \text{for every $t\in I$}.$$
	 \medskip
	 \noindent (iii)\,\,We split the proof of this statement into two steps. \medskip
	 
	 \textsc{Step I:} We begin by showing that, if $N > 0$ is as in \eqref{eq.choiceN}, then
	 \begin{equation} \label{eq.minAxStepI}
	  \min_{t\in I}\big|\mathcal{A}_x(t)\big| \leq N.
	  \end{equation}
	 We argue again by contradiction and, to fix ideas, we assume that 
	 \begin{equation} \label{eq.todeducexprimegeq}
	  \mathcal{A}_x(t) = a\big(t,x(t)\big)\,x'(t) > N \quad \text{for a.e.\,$t\in I$}.
	  \end{equation}
	 By integrating both sides of this inequality on $[0,T]$, we get
	 $$\int_0^T\mathcal{A}_x(t)\,\d t = \int_0^Ta(t,x(t))\,x'(t)\,\d t > NT;$$
	 from this, by statement (ii), \eqref{eq.defa0ash2} and the choice of $N$ in 
	 \eqref{eq.choiceN}, we obtain
	 (notice that, by \eqref{eq.todeducexprimegeq}, we have $x'(t) > 0$ a.e.\,on $I$)
	 \begin{equation*}
	 \begin{split}
	  & NT < \int_0^Ta(t,x(t))\,x'(t)\,\d t 
	  \leq a_0\cdot\int_0^Tx'(t)\,\d t \\[0.2cm]
	  & \quad = (\nu_2-\nu_1)\cdot a_0 
	  = |\nu_2-\nu_1|\cdot a_0 \leq (2M)\cdot a_0
	  < NT.
	 \end{split}
	 \end{equation*}
	 This is clearly a contradiction, hence $\mathcal{A}_x\leq N$ on $I$. By arguing analogously
	 one can also prove that $\mathcal{A}_x\geq - N$ on $I$, and \eqref{eq.minAxStepI}
	 is established.
	 
	 \medskip
	 
	 \textsc{Step II:} We now turn to prove statement (iii). To this end,
	 arguing once again by contradiction, we assume that there exists $\overline{t}\in I$ such that
	 $$\big|\mathcal{A}_x(\overline{t})\big| > L_M;$$
	 moreover, to fix ideas, we suppose that 
	 $\mathcal{A}_x(\overline{t}) > L_M$. \medskip
	 
	 Since, by definition, $L_M > N$, from Step I and  
	 Remark \ref{rem.contax} we 
	 infer the existence of two points $t_1,t_2\in I$, with 
	 $t_1 < t_2$, such that (for example) \medskip
	 
	 $(*)$\,\,$\mathcal{A}_x(t_1) = N$ and $\mathcal{A}_x(t_2) = L_M$; \medskip
	 
	 $(**)$\,\,$0 < N < \mathcal{A}_x(t) < L_M$ for every $t\in (t_1,t_2)$; \medskip
	 
	 \noindent from this, by statement (ii), \eqref{eq.defa0ash2} and the choice of 
	 $N$ (see \eqref{eq.choiceN}), we obtain
	\begin{align} \label{eq.mainestimtoabsurd}
	0 < H <
	\frac{N}{a_0} <
	x'(t) < \frac{L_M}{h(t)} = \gamma_0(t)
	\leq \hat{\gamma}(t) \quad \text{for a.e.\,$t\in(t_1,t_2)$}.
	\end{align}	 
	 Now, by definition of $\mathcal{D}$, we deduce from 
	 \eqref{eq.mainestimtoabsurd} that
	 $\mathcal{D}(x') = x'$ a.e.\,on 
	 $(t_1,t_2)$; moreover, by statement (i)
	 and the definition of $f^*$, we have
	 $$f^*\Big(t, x(t), \mathcal{D}\big(\mathcal{T}(x)'(t)\big)\Big)  
	 = f(t,x(t),x'(t)) \quad \text{for a.e.\,$t\in(t_1,t_2)$}.$$
	 As a consequence, since $x(t)\in[\alpha(t),\beta(t)]$ for every $t\in I$
	 (by statement (i)) and since $x'(t) > H > 0$ for a.e.\,$t\in (t_1,t_2)$
	 (again by \eqref{eq.mainestimtoabsurd}), we are entitled
	 to apply estimate \eqref{eq.Nagumocondition}, 
	 which gives (remind that $x$ solves \eqref{eq.BVPtrunc} and see $(**)$)
	 \begin{align*}
	  & \bigg|\Big(\Phi\big(\mathcal{A}_x(t)\big)\Big)'\bigg| = \dsy
	  \bigg|\Big(\Phi\big(a(t,x(t))\,x'(t)\big)\Big)'\bigg|
	  = \big|f(t,x(t),x'(t))\big| \\[0.1cm]
	  & \qquad\qquad\leq 
	  \psi\Big(|\Phi(a(t,x(t))\,x'(t))|\Big)
	 \cdot\Big(l(t)+\mu(t)\,(x'(t))^{\frac{q-1}{q}}\Big) \\[0.1cm]
	  & \qquad\qquad =
	  \psi\Big(\big|\Phi\big(\mathcal{A}_x(t)\big)\big|\Big)\cdot 
	  \Big(l(t)+\mu(t)\,(x'(t))^{\frac{q-1}{q}}\Big) \\[0.1cm]
	  & \qquad\qquad (\text{since, by $(**)$ and \eqref{eq.choiceN}, we have
	  $\Phi(\mathcal{A}_x(t)) > \Phi(N) > 0$}) \\[0.1cm]
	  & 
	  \qquad\qquad =
	  \psi\Big(\Phi\big(\mathcal{A}_x(t)\big)\Big)\cdot 
	  \Big(l(t)+\mu(t)\,(x'(t))^{\frac{q-1}{q}}\Big)
	 \qquad \big(\text{a.e.\,on $(t_1,t_2)$}\big).
	 \end{align*}
	 In particular, by exploiting this inequality, we obtain 
	 \begin{align*}
	 & \int_{\Phi(N)}^{\Phi(L_M)}\frac{1}{\psi(s)}\,\d s
	    = \int_{\Phi(\mathcal{A}_x(t_1))}^{\Phi(\mathcal{A}_x(t_2))}
	    \frac{1}{\psi(s)}\,\d s \\[0.2cm]
	   & \qquad \leq 
	 \int_{t_0}^{t_1}\frac{\big(\Phi\big(\mathcal{A}_x(t)\big)\big)'}
	   {\psi\big(\Phi\big(\mathcal{A}_x(t)\big)\big)}\,\d t
	   \leq \int_{t_0}^{t_1}
	   \Big(l(t)+\mu(t)\,(x'(t))^{\frac{q-1}{q}}\Big)\,\d t \\[0.2cm]
	   & \qquad \big(\text{by H\"older's inequality}\big) \\
	   & \qquad
	   \leq \|l\|_{L^1} + \|\mu\|_{L^q}\cdot
	   \bigg(\int_{t_0}^{t_1} x'(t)\,\d t\bigg)^{\frac{q-1}{q}} \\
	   & \qquad \leq
	   \|l\|_{L^1} + \|\mu\|_{L^q}\cdot\big(x(t_1)-x(t_0)\big)^{\frac{q-1}{q}} \\[0.2cm]
	   & \qquad \big(\text{by statement (ii)}\big) \\
	   & \qquad \leq
	   \|l\|_{L^1} + \|\mu\|_{L^q}\cdot(2M)^{\frac{q-1}{q}}.
	 \end{align*}
	 This is in contradiction with the choice of $L_M$ (see \eqref{eq.choiceL2}), 
	 hence $\mathcal{A}_x \leq L_M$ on $I$. Analogously, one can
	 prove that $\mathcal{A}_x\geq -L_M$ on $I$ 
	 and statement (iii) is completely proved. \medskip

	 \noindent (iv)\,\,From statement (iii) and assumption (A2) we immediately infer that
	 $$|x'(t)|\leq \frac{L_M}{a(t,x(t))} \leq \frac{L_M}{h(t)}=\gamma_0(t) 
	 \quad \text{for almost
	 every $t\in I$},$$
	 and the proof is finally complete.  
	\end{proof}
	By combining Propositions \ref{prop.existenceBVPtrunc} 
	and \ref{prop.equivalenceBVPs}, we can prove Theorem \ref{thm.mainconcrete}.
	\begin{proof} [Proof (of Theorem \ref{thm.mainconcrete})]
	 First of all, by Proposition \ref{prop.existenceBVPtrunc}, there exists
	 (at least) one solution $x\in W^{1,p}(I)$ of the ``truncated'' Dirichlet
	 problem \eqref{eq.BVPtrunc};
	 mo\-re\-o\-ver, by statements (i) and (iii)
	 of Proposition \ref{prop.equivalenceBVPs} (and the very definitions
	 of the operators $\mathcal{T}$ and $\mathcal{D}$), for almost every 
	 $t\in I$ we obtain
	 \begin{align*}
	  & \bigg(\Phi\Big(a\big(t,x(t)\big)\,x'(t)\Big)\bigg)'
	  = \bigg(\Phi\Big(a\big(t,\mathcal{T}(x)(t)\big)\,x'(t)\Big)\bigg)'
	  \\
	  & \qquad\qquad = 
	  f^*\Big(t, x(t), \mathcal{D}\big(\mathcal{T}(x)'(t)\big)\Big)  
	 = f(t,x(t),x'(t)).
	 \end{align*}
	 Thus, $x$ is actually a solution of the Dirichlet problem \eqref{eq.mainBVPconcrete}.
	 To complete the demonstration of the theorem, we show that $x$ satisfies 
	 \eqref{eq.xbetweenUL} and
	\eqref{eq.estimsolML}. 
	
	As for \eqref{eq.xbetweenUL}, it is precisely
	statement (i) of Proposition \ref{prop.equivalenceBVPs};
	estimate \eqref{eq.estimsolML}, instead, follows from
	 statements (ii) and (iii) of the same proposition.  
	\end{proof}
	\medskip
 \paragraph{Some examples.}
 We close the section with a few illustrating examples, 
 in which we consider a generic function $a(t,x)$ satisfying assumption (A2).
 We explicitly point out that (A2) is verified, e.g., in the following special cases:
 \begin{itemize}
 \item[{(1.)}] when $a(t,x)$ has a product structure
	$$a(t,x)=h(t)\cdot b(x),$$
	where $h:I\to\R$ is a continuous non-negative function on $I$ such that 
 	$1/h$ is in  $L^p(I)$ (for some $p> 1$) and
 	$b\in C(\R)$ is such that $\inf_\R b > 0$; \vspace*{0.08cm}
 \item[(2.)] when $a(t,x)$ is a sum
 $$a(t,x)=h(t)+b(x),$$
	where $h:I\to\R$ is a continuous non-negative function on $I$ such that 
 	$1/h$ is in $L^p(I)$ (for some $p> 1$) and
 	$b\in C(\R,\R)$ is non-negative.
\end{itemize} 
\noindent 
In the next Example \ref{example1}, the growth of the right-hand side $f$ 
with respect the variable $y$ is linear, and this allows
the choice $\psi \equiv 1$  in the Wintner-Nagumo condition \eqref{eq.Nagumocondition}.
Thus, condition \eqref{eq.Nagumocondition} does not require any relation 
among the differential operator $\Phi$, 
the function $a$ appearing inside $\Phi$, and  $f$.

\begin{ex} \label{example1}
Let us consider the Dirichlet problem
\begin{equation} \label{e:example1}
 \begin{cases} 
  \dsy\Big(\Phi\big(a(t,x(t))\,x'(t)\big)\Big)' = 
  \sigma(t)(x(t)+ \rho(t)) + g(x(t))\,x'(t) \\[0.2cm] 
  x(0) = \nu_1,\,\,x(T) = \nu_2,
  \end{cases}
 \end{equation}
 where $\varphi,\,a,\,\sigma,\,\rho$ and $g$ satisfy the following assumptions: \medskip
 
 $(\star)$\,\,$\Phi:\R \to \R$ is a generic strictly increasing homeomorphism; \medskip
 
 $(\star)$\,\,$a\in C(I\times\R,\R)$ satisfies assumption (A2); \medskip
 
 $(\star)$\,\,$\sigma \in L^1(I)$ and $\sigma\geq 0$ a.e.\,on $I$; \medskip
 
 $(\star$)\,\,$\rho\in C(I)$ and $g\in C(\R,\R)$ are generic. \medskip

\noindent We aim to show that our Theorem \ref{thm.mainconcrete} can be applied to
	problem \eqref{e:example1}.
  To this end, we consider the function $f$ defined as follows:
 $$f:I\times\R^2\to\R, \qquad f(t,x,y):= \sigma(t)(x+\rho(t))+ g(x)y.$$
 Obviously, $f$ is a Carath\'eodory function; moreover, it is very
 easy to recognize that $f$ satisfies assumption (A2)'. 
 Indeed, let $R > 0$ be arbitrarily fixed and let $\gamma$ be a non-negative
 function in $L^1(I)$; setting 
 $$M_R := \max_{[-R,R]}|g|,$$ 
 we then have
 $$\big|f(t,x,y(t))\big| \leq \sigma(t)\,\big(R+|\rho(t)|\big) + 
  M_R\cdot \gamma(t)=: h_{R,\gamma}(t),$$
 for every $x\in\R$ with $|x|\leq R$ and every $y\in L^1(I)$ 
 satisfying
 $|y(t)|\leq \gamma(t)$ for almost every $t\in I$. 
 Since the function
 $h_{R,\gamma}$ is non-negative and belongs to $\in L^1(I)$
 (by the assumptions on $\sigma,\,\rho$ and $\gamma$), we conclude that
 $f$ fulfills \eqref{eq.estimfh} in assumption (A2)', as claimed. 
 
  We now observe that, setting $N := \max_{I} |\rho|$, the constant functions
  $$\alpha(t) := -N \qquad
  \beta(t) := N \qquad \big(\text{for $t\in I$}$$
  are, respectively, a lower and a upper solution of
  \eqref{eq.maiODEtocite} such that
  $\alpha\leq \beta$ on $I$; hence,
  assumption (A1)' is satisfied. Furthermore,
  since we have
  $$|f(t,x,y)|\,\leq\,(2 N)\,\sigma(t)+ \Big(\max_{x\in [-N,N]}|g(x)|\Big)\cdot |y|$$
  for every $t\in I$, every
  $|x|\leq N$ and every $y\in\R$, we conclude that $f$
  also satisfies assumption (A3)' with the choice
  (here, $M_N := \max_{[-N,N]}|g|$)
  $$H := 1, \quad \psi \equiv 1, \quad l(t) := 2 N\,\sigma(t), \quad
  \mu(t):= M_N\ \quad \text{and} \quad q=\infty.$$
  We are then entitled to apply Theorem \ref{thm.mainconcrete}, which ensures that
  there exists (at least) one solution of problem \eqref{e:example1}
  for every fixed $\nu_1,\,\nu_2\in [-N,N]$.
\end{ex}
 In the next Example 
 \ref{exm.rLapl} we provide an application of Theorem \ref{thm.mainconcrete} 
 for a rather general right-hand side, with possible superlinear g
 rowth wrt $u'$.
\begin{ex} \label{exm.rLapl}
 Let us consider the following Dirichlet problem
 \begin{equation} \label{e:example3}
  \begin{cases} 
  \Big(\Phi_r\big(a(t,x(t))\,x'(t)\big)\Big)' = \sigma(t)\cdot g(x(t))\cdot|x'(t)|^{\delta} \\[0.2cm]
   u(0) = \nu_1,\,\,u(T) = \nu_2,
  \end{cases} 
 \end{equation}
 where $\Phi_r,\,a,\,\sigma,\,g$ and the exponent $\delta$ satisfy the following
 assumptions: \medskip
 
 $(\star)$\,\,$\Phi_r:\R \to \R$ is the standard $r$-Laplacian, that is, 
 $$\Phi_r(\xi):= |\xi|^{r-2}\cdot\xi
 \qquad \big(\text{for a suitable $r > 1$}\big);$$
 
 $(\star)$\,\,$a\in C(I\times\R,\R)$  satisfies assumption (A2), that is, 
 \begin{flalign}
  & \qquad\qquad 
  \bullet\,\,\text{$h \geq 0$ on $I$ and $1/h\in L^p(I)$ (for some $p > 1$)}; && \nonumber \\[0.15cm]
  & \qquad\qquad 
  \bullet\,\,\text{$a(t,x) \geq h(t)$ for every $t\in I$ and every $x\in\R$}. \nonumber
 \end{flalign}
 
 $(\star)$\,\,$\sigma \in L^\tau(I)$  for a suitable $\tau > 1$ satisfying the relation
 \begin{equation} \label{e:alphabeta2}
  \frac{1}{\tau} + \frac{r-1}{p} < 1;
 \end{equation}  
 
 $(\star)$\,\,$g\in C(\R,\R)$ is a generic function; \medskip
 
 $(\star)$\,\,$\delta$ is a positive real constant satisfying the relation
 \begin{align} \label{e:alphabeta1}
  \delta \leq 1-\frac{1}{\tau} + (r-1)\,\left(1-\frac{1}{p}\right).
 \end{align}  
 We aim to show that our Theorem \ref{thm.mainconcrete} can be applied
 to problem \eqref{exm.rLapl}.
 To this end, we consider the function $f$ defined as follows:
 $$f:I\times\R^2\to\R, \qquad f(t,x,y):= \sigma(t)\cdot g(x)\cdot |y|^\delta.$$
 Obviously, $f$ is a Carath\'eodory function; moreover, it is not difficult
 to recognize that $f$ satisfies assumption (A2)'. 
 Indeed, let $R > 0$ be ar\-bi\-tra\-ri\-ly fixed and let $\gamma$ be a non-negative
 function in $L^1(I)$; setting 
 $$M_R := \max_{[-R,R]}|g|,$$ 
 we then have
 $$\big|f(t,x,y(t))\big| \leq M_R\cdot|\sigma(t)|\cdot(\gamma(t))^{\delta} =: h_{R,\gamma}(t)$$
 for every $|x|\leq R$ and every $y\in L^1(I)$ such that $|y(t)|\leq \gamma(t)$
 a.e.\,on $I$. Now, by combining \eqref{e:alphabeta2} with \eqref{e:alphabeta1} we readily
 see that
 \begin{equation} \label{ab2}
	\delta< \left(1-\frac1\tau\right)p;
  \end{equation}
  from this, by H\"older's inequality
  (and the assumptions on $\sigma$ and $\gamma$), we infer that
  $h_{R,\gamma}$ (which is non-negative on $I$) belongs to $L^1(I)$, whence
  $f$ satisfies (A2)'.
  In order to prove that also assumptions (A1)' and (A3)' are satisfied
  we first notice that, if $N > 0$ is \emph{arbitrary},
  the constant functions
  $$\alpha(t) := -N \qquad
  \beta(t) := N \qquad \big(\text{for $t\in I$}$$
  are, respectively, a lower and a upper solution of
  \eqref{eq.maiODEtocite} such that
  $\alpha\leq \beta$ on $I$; hence,
  assumption (A1)' is fulfilled. Moreover, by 
  \eqref{e:alphabeta1} we have
  $$\delta \leq (r-1) + \frac{q-1}{q},
  \quad \text{where}\,\,q:=\frac{\tau\,p}{p+\tau\,(r-1)} > 1.$$
  From this, setting 
  $M_N:={\max_{[-N,N]}|g|}$, we obtain
  \begin{align*}
   \big|f(t,x,y)\big| & \leq M_N\cdot|\sigma(t)|\cdot|y|^{\delta} \leq M_N\cdot|\sigma(t)|\cdot|y|^{r-1}
   \cdot |y|^\frac{q-1}{q} \\[0.2cm]
  &  = \big|\Phi(a(t,x)y)\big|\cdot  
  \bigg(\frac{M_N\cdot|\sigma(t)|}{(a(t,x))^{r-1}}\bigg)\,|y|^{\frac{q-1}{q}} \\[0.2cm]
  & \big(\text{since $a(t,x)\geq h(t)$ for every $(t,x)\in I\times\R$}\big) \\[0.2cm]
   & \leq 
   \big|\Phi(a(t,x)y)\big|\cdot
  \bigg(\frac{M_N\cdot|\sigma(t)|}{(h(t))^{r-1}}\bigg)\,|y|^{\frac{q-1}{q}}  
  \end{align*}
  for a.e.\,$t\in I$, every $x\in\R$ with
  $|x|\leq N$ and every $y\in\R$ satisfying $|y|\geq 1$. Thus, if we are able to prove that
  \begin{equation} \label{eq.muLqrLapltoprove}
   t\mapsto\frac{|\sigma(t)|}{(h(t))^{r-1}} \in L^q(I),
   \end{equation}
  we can conclude that $f$ satisfies assumption (A3)' with the choice
  $$H := 1, \quad \psi(s) := s, \quad l(t) := 0, \quad
  \mu(t):= \frac{M_N\cdot|\sigma(t)|}{(h(t))^{r-1}}$$
  and $q$ as above. On the other hand, the needed
  \eqref{eq.muLqrLapltoprove} is an easy consequence of
  H\"older's inequality, assumption \eqref{e:alphabeta2} and of the fact
  that $1/h\in L^p(I)$.
  
  We are then entitled to apply Theorem \ref{thm.mainconcrete}, which ensures the existence
  of (at least) one solution of the Dirichlet
  problem \eqref{e:example3} \emph{for every $\nu_1\,\nu_2\in\R$}.
\end{ex}

	\section{General nonlinear boundary conditions} \label{sec:general}
	The main aim of this last section is to prove the solvability
	of \emph{general boundary value problems}
	associated with the (possibly singular) differential equation
	\begin{equation} \label{eq.generalODEgeneralBVP}
	 \dsy\Big(\Phi\big(a(t,x(t))\,x'(t)\big)\Big)' = f(t,x(t),x'(t)), \qquad \text{a.e.\,on $I$}.
	 \end{equation}	
	 (here, $\Phi,\,a$ and $f$ satisfy the assumptions (A1)-to-(A3)
	introduced in Section \ref{sec:concrete}).
	As a particular case, we shall obtain existence
	results for \emph{periodic BVPs} and for \emph{Sturm-Liouville-type problems}
	(associated with \eqref{eq.generalODEgeneralBVP}). \medskip
	
	To be more precise, taking for fixed all the notation
	introduced so far, we aim to study
	the following general BVPs
	(associated with \eqref{eq.generalODEgeneralBVP}):
	\begin{equation} \label{eq.maingeneralBVPgh}
	 \begin{cases}
	 \dsy \dsy\Big(\Phi\big(a(t,x(t))\,x'(t)\big)\Big)' = 
	 f(t,x(t),x'(t)), \qquad \text{a.e.\,on $I$}, \\[0.2cm]
	 g(x(0),x(T),\mathcal{A}_x(0),\mathcal{A}_x(T)) = 0, \\[0.2cm]
	 x(T) = \rho(x(0)).
	 \end{cases}
	\end{equation}
	Here, $w:\R\to\R$ and $g:\R^4\to\R$ satisfy
	the following general assumptions:
	\begin{itemize}
	 \item[(G1)] $\rho\in C(\R,\R)$ and is increasing on $\R$; \medskip
	 \item[(G2)] $g\in C(\R^4,\R)$ and, for every fixed $u,v\in\R$, it holds that \medskip
	 \begin{itemize}
	  \item[$\mathrm{(G2)}_1$] $g(u,v,\cdot,z)$ is increasing for every fixed $z\in\R$; \vspace*{0.15cm}
	  \item[$\mathrm{(G2)}_2$]	$g(u,v,w,\cdot)$ is decreasing for every fixed $w\in\R$.
	 \end{itemize}
	\end{itemize}
	We now state one of the main existence results of this section.
	\begin{thm} \label{thm.maingeneralBVP}
	Let us assume that \emph{all} the hypotheses of Theorem \ref{thm.mainconcrete} are
	satisfied, and that $g$ and $h$ fulfill the assumptions \emph{(G1)-(G2)}
	introduced above.	
	Moreover, if $\alpha,\beta \in W^{1,p}(I)$ are as in assumption \emph{(A1')}, we suppose that
	\begin{equation} \label{eq.assumptionlowerupper}
	 \begin{cases}
	 g(\alpha(0),\alpha(T),\mathcal{A}_\alpha(0),\mathcal{A}_\alpha(T)) \geq 0, \\[0.2cm]
	 \alpha(T) = \rho(\alpha(0))
	 \end{cases}\,\,
	 \begin{cases}
	 g(\beta(0),\beta(T),\mathcal{A}_\beta(0),\mathcal{A}_\beta(T)) \leq 0, \\[0.2cm]
	 \beta(T) = \rho(\beta(0)).
	 \end{cases}
	\end{equation}
	Finally, let us assume that the function
	$a$ satisfies the following
	condition:
	$$a(0,x) \neq 0 \quad \text{and} \quad a(T,x) \neq 0 \qquad \text{for every $x\in\R$}.$$
	Then the problem \eqref{eq.maingeneralBVPgh}
	possesses one solution $x\in W^{1,p}(I)$ such that
	\begin{equation} \label{eq.solutionxfunctionalinterval}
	 \alpha(t)\leq x(t)\leq\beta(t) \qquad\text{for every $t\in I$}.
	\end{equation}
	Furthermore, if $M > 0$ is any real number
	such that $\sup_I|\alpha|,\,\sup_I|\beta|\leq M$
	and $L_M > 0$ \emph{is as in Theorem \ref{thm.mainconcrete} (see \eqref{eq.choiceL2})},
    one has
   \begin{align}
    \max_{t\in I}\big|x(t)\big|\leq M \quad \text{\emph{and}} \quad
    \max_{t\in I}\big|\mathcal{A}_x(t)\big| \leq L_M. \label{eq.estimsolMLgeneralgh}
   \end{align}
	\end{thm}
	The basic idea behind the proof of Theorem \ref{thm.maingeneralBVP}, inspired by the work
	of \cite{CabadaPo} and already exploited by \cite{MaPa2017}, is
	to think of the boundary value problem \eqref{eq.maingeneralBVPgh} as a ``superposition''
	of Dirichlet problems to which our existence result in Theorem \ref{thm.mainconcrete} apply.
	Following this approach, we first establish a com\-pac\-tness\--type result
	for the solutions of the ODE \eqref{eq.generalODEgeneralBVP}.
	\begin{prop} \label{prop.compactness}
     For every $n\in\N$, let $x_n\in W^{1,p}(I)$
	 be a solution of 
	 \begin{equation} \label{eq.equationsolvesequence}
	  \Big(\Phi\big(a(t,x(t))\,x'(t)\big)\Big)' = f(t,x(t),x'(t)), \qquad \text{a.e.\,on $I$}.
	 \end{equation}
	 We assume that, together with \emph{(A1)}-to-\emph{(A3)}, $f$ satisfies assumption
	 $\mathrm{(A2')}$ of Theo\-rem \ref{thm.mainconcrete}; moreover, we suppose
	  that there exist $M,L > 0$ such that
	 \begin{equation} \label{eq.assumptionxAxsequence}
	 \sup_I|x_n|\leq M \quad \text{and} \quad
	 \sup_I|\mathcal{A}_{x_n}| \leq L \qquad \text{for every $n\in\N$}.
	 \end{equation}
	 It is then possible to
	 find a sub-sequence $\{x_{n_k}\}_{k\in\N}$ of $\{x_n\}_{n\in\N}$ 
	 and a solution $x_0\in W^{1,p}(I)$ of the equation \eqref{eq.generalODEgeneralBVP}
	 with the following properties:
	 \begin{itemize}
	  \item[\emph{(1)}] $x_{n_k}(t)\to x_0(t)$ for every $t\in I$ as $n\to\infty$;
	  \item[\emph{(2)}] $\mathcal{A}_{x_{n_k}}(t)\to\mathcal{A}_{x_0}(t)$ for every $t\in I$ as $n\to\infty$.
	 \end{itemize}
	\end{prop}
	\begin{proof}
	 For every natural $n$, we set $z_n := \big(\Phi(\mathcal{A}_{x_n})\big)'$.
	 Since $x_n$ is a solution of \eqref{eq.equationsolvesequence}, by 
	 \eqref{eq.assumptionxAxsequence} and the fact that $f$
	 satisfies $\mathrm{(A2')}$, we have
	 (see also \eqref{eq.estimfh})
	 \begin{equation} \label{eq.touseattheend}
	  \big|z_n(t)\big| = \big|f(t,x_n(t),x'_n(t))\big|
	 \leq h_{M,\gamma}(t) \quad \text{for a.e.\,$t\in I$},
	 \end{equation}
	 where, $h_{M,\gamma}\in L^p(I)$ is the function appearing in assumption
	 $\mathrm{(A2')}$ and cor\-re\-spon\-ding to $M$ and $\gamma = L/h$.
	 Moreover, again by \eqref{eq.assumptionxAxsequence}, one has
	 \begin{equation} \label{eq.derxntouseLebesgue}
	  |x'_n(t)| \leq \frac{L}{h(t)} = \gamma(t) \quad \text{for a.e.\,$t\in I$}.
	  \end{equation}
	 As a consequence, both $\{z_n\}_n$ and $\{x'_n\}_n$ are 
	 \emph{uniformly integrable} in $L^1(I)$. 	
	 Then, by Dunford-Pettis Theorem (see, e.g., \cite{Brezis}),
	 there exist $v,w\in L^1(I)$ such that,
	 up to a sub-sequence,
	 $x'_n\rightharpoonup v$ and $z_n\rightharpoonup w$ in $L^1(I)$ as $n\to\infty$.
	 
	 Now, since the sequence $\{x_n(0)\}_n$ is bounded in $\R$ (again
	 by \eqref{eq.assumptionxAxsequence}), we can assume that
	 $x_n(0)$ converges to some $\nu_0\in\R$ as $n\to\infty$; from this,
	 reminding that $x'_n \rightharpoonup v$ in $L^1(I)$, we get
	 \begin{equation} \label{eq.xnconvergesx0lemma}
	  x_n(t) = x_n(0)+\int_0^tx'_n(s)\,\d s \underset{n\to\infty}{\longto} 
	  \nu_0 + \int_0^t v(s)\,\d s =: x_0(t)
	 \quad \text{$\forall\,t\in I$}.
	 \end{equation}
	 Notice that, by its very definition, $x_0$ satisfies the following properties: \medskip
	 
	 (a)\,\,$x_0$ is absolutely continuous on $I$ and $x_0' = v \in L^1(I)$; \medskip
	 
	 (b)\,\,$\sup_I|x_0| \leq M$ (this follows also from \eqref{eq.assumptionxAxsequence}) . \medskip
	 
	 \noindent Thus, to complete the demonstration, we 
	 need to prove that $x_0$ is a solution of
	 the equation
	 \eqref{eq.equationsolvesequence} and that $\mathcal{A}_{x_n}\to \mathcal{A}_{x_0}$
	 point-wise on $I$ as $n\to\infty$. \medskip

	First of all, since also the sequence $\{\mathcal{A}_{x_n}(0)\}_n$ is bounded
	 in $\R$ (again by \eqref{eq.assumptionxAxsequence}), we can suppose that
	 $\mathcal{A}_{x_n}(0)\to \nu'_0$ as $n\to\infty$ for some $\nu'_0\in\R$;
	 thus, since $z_n\rightharpoonup w$ in $L^1(I)$, we have
	 \begin{equation*}
	  \Phi\big(\mathcal{A}_{x_n}(t)\big) = \Phi\big(\mathcal{A}_{x_n}(0)\big)
	  + \int_0^t z'_n(s)\,\d s \underset{n\to\infty}{\longto} 
	  \Phi(\nu'_0)+\int_0^t w(s)\,\d s.
	 \end{equation*}
	 As a consequence, by the continuity of $\Phi^{-1}$, we obtain
	 \begin{equation}  \label{eq.Axnconvergessthlemma}
	  \mathcal{A}_{x_n}(t) 
	\underset{n\to\infty}{\longto} 	 
	  \Phi^{-1}\bigg(\Phi(\nu'_0) + \int_0^t w(s)\,\d s\bigg) =: \mathcal{U}(t)
	 \quad \text{for every $t\in I$}.
	 \end{equation}
	 Notice that, by definition, $\mathcal{U}\in C(I,\R)$ and it
	 satisfies \medskip
	 
	 $\mathrm{(a)}_1$\,\,$\Phi\circ\mathcal{U}$ is absolutely continuous
	 on $I$ and $(\Phi\circ\mathcal{U})' = w\in L^1(I)$; \medskip
	 
	 $\mathrm{(b)}_1$\,\,$\sup_I|\mathcal{U}| \leq L$ 
	 (this follows also from \eqref{eq.assumptionxAxsequence}) . \medskip
	 
	 \noindent Now, since $a$ is continuous on $I\times\R$, we derive from
	 \eqref{eq.xnconvergesx0lemma} that
	 $a(t,x_n(t))$ converges to $a(t,x_0(t))$ for every $t\in I$ as $n\to\infty$;
	 thus, the above \eqref{eq.Axnconvergessthlemma} 
	 (together with the fact that $a(t,x)\geq h(t) > 0$ for a.e.\,$t\in I$)
	 gives
	 \begin{equation} \label{eq.derxnconvergesderx}
	  x'_n(t) \underset{n\to\infty}{\longto} 
	 \frac{1}{a(t,x_0(t))}\,\mathcal{U}(t)
	 \quad \text{for a.e.\,$t\in I$}.
	 \end{equation}
	 Taking into account \eqref{eq.derxntouseLebesgue} and the fact that
	 $1/h\in L^p(I)$ (see assumption (A2)),
	 it is easy to recognize that 
	 $$x'_n\to \frac{\mathcal{U}}{a(\cdot,x_0(\cdot))}\qquad \text{also in $L^1(I)$};$$
	 on the other hand, since we already
	  know that $x'_n\rightharpoonup v$ in $L^1(I)$ as $n\to\infty$, we necessarily have
	 \begin{equation} \label{eq.gcoincidesUfraca}
	  v(t) = \frac{1}{a(t,x_0(t))}\,\mathcal{U}(t)	 \quad \text{a.e.\,on $I$}.
	  \end{equation}
	 From this, by reminding that $v = x_0'$ (see (a)), we infer that \medskip
	 
	 $(\star)$\,\,$x_0' = v\in L^p(I)$, whence $x_0\in W^{1,p}(I)$
	 (as $|{\mathcal{U}}/{a(\cdot,x_0(\cdot))}|\leq L/h$); \medskip
	 
	 $(\star)$\,\,$a(t,x_0)\,x'_0 = \mathcal{U}$ a.e.\,on $I$; \medskip
	 
	 $(\star)$\,\,$\Phi\circ\big(a(t,x_0)\,x_0'\big) =
	 \Phi \circ \mathcal{U} \in W^{1,1}(I)$
	 and (see $\mathrm{(a)_1}$)
	 $$\big(\Phi(a(t,x_0)\,x_0')\big)' = w.$$
	 We now turn to prove that $x_0$ solves the ODE
	 \eqref{eq.equationsolvesequence}. To this end we 
	 observe that, by \eqref{eq.derxnconvergesderx} and 
	 \eqref{eq.gcoincidesUfraca}, we have
	 $x_n'(t)\to v(t) = x_0'(t)$ for a.e.\,$t\in I$; as a consequence,
	 since $x_n$ is a solution of \eqref{eq.equationsolvesequence}
	 for every $n$ and $f$ is a Carath\'{e}odory function
	 (see (A3)), we obtain (remind that $x_n\to x_0$ point-wise on $I$)
	 $$z_n = \Big(\Phi\big(\mathcal{A}_{x_n}\big)\Big)'
	 = f(t,x_n(t),x_n'(t)) 
		\underset{n\to\infty}{\longto} 	  
	   f(t,x_0(t),x'_0(t)) \quad \text{for a.e.\,$t\in I$}.
	  $$
	  On the other hand, by \eqref{eq.touseattheend}, we have that
	  $z_n\to f(t,x_0(t),x_0'(t)$ also in $L^1(I)$; since we already know that $
	  z_n\rightharpoonup w$ in $L^1(I)$, we conclude that
	  $$\big(\Phi(a(t,x_0(t))\,x_0'(t))\big)' = w(t)
	  = f(t,x_0(t),x_0'(t)) \quad \text{for a.e.\,$t\in I$},$$
	  that is, $x_0$ is a solution of \eqref{eq.equationsolvesequence}.
	  Finally, since $\mathcal{U}$ is a continuous function on $I$ such that
	  $\mathcal{U} = a(t,x_0)\,x_0'$ a.e.\,on $I$, we have
	  $\mathcal{U} = \mathcal{A}_x$ on $I$ and, by \eqref{eq.Axnconvergessthlemma},
	  $$\mathcal{A}_{x_n}(t) 
		\underset{n\to\infty}{\longto} 	  
		\mathcal{A}_x(t) \quad \text{for every $t\in I$}.$$
		This ends the proof.  
	 \end{proof}
	 We also need the following technical lemma.
	 \begin{lem} \label{lem.technicalAx}
	  Let $\alpha,\,\beta\in W^{1,p}(I)$ be, respectively,
	  a lower and a upper solution of the equation \eqref{eq.generalODEgeneralBVP}
	  such that $\alpha\leq\beta$. Moreover,
	  let us assume that
	  $$a(0,x) \neq 0 \quad \text{and} \quad a(T,x) \neq 0 \qquad\text{for every $x\in\R$}.$$
	  Then the following facts hold true: 
	  \begin{itemize}
	   \item[\emph{(i)}] if $\alpha(0) = \beta(0)$, then $\mathcal{A}_\alpha(0) \leq \mathcal{A}_\beta(0)$;
	   \item[\emph{(ii)}] if $\alpha(T) = \beta(T)$, then $\mathcal{A}_\alpha(T)\geq\mathcal{A}_\beta(T)$.
	  \end{itemize}
	 \end{lem}
	 \begin{proof}
	 We only prove statement (i), since (ii) is analogous.
	 \medskip
	 
	 First of all,
	 since both $a(0,\alpha(0))$ and $a(0,\beta(0))$ are different from $0$
	 (by as\-sum\-ption), it is possible to find $\delta > 0$ such that,
	 for a.e.\,$t\in[0,\delta]$, we have
	 $$\alpha'(t) = \frac{\mathcal{A}_{\alpha}(t)}{a(t,\alpha(t))}=:u_1(t) \quad
	   \text{and}\quad
	   \beta'(t) = \frac{\mathcal{A}_{\beta}(t)}{a(t,\beta(t))} =: u_2(t)$$
	 moreover, both $u_1$ and $u_2$ are continuous on $[0,\delta]$.
	 Let us now assume, by contradiction, that $\mathcal{A}_\alpha(0) > \mathcal{A}_\beta(0)$.
	 Since, by assumption, $\alpha(0) = \beta(0)$ (and $a$ is non-negative on $I\times\R$), 
	 there exists $\delta' < \delta$ such that
	 $$\alpha'(t) = u_1(t) > u_2(t) = \beta'(t) \quad \text{for a.e.\,$t\in[0,\delta']$};$$
	 thus, by integrating this inequality on $[0,\delta']$, we get
	 \begin{align*}
	 \alpha(t) & = \alpha(0) + \int_0^t\alpha'(s)\,\d s
	 = \beta(0) + \int_0^t\alpha'(s)\,\d s \\
	 & > \beta(0) + \int_0^t\beta'(s)\,\d s = \beta(t) \qquad (\text{for every $t\in[0,\delta']$},
	 \end{align*}
	 which contradicts the fact that $\alpha\leq \beta$ on $I$. This ends the proof.  
	 \end{proof}
   	  We can now prove Theorem
	 \ref{thm.maingeneralBVP}.
	 \begin{proof} [Proof (of Theorem \ref{thm.maingeneralBVP})] 
	  Let $\nu\in[\alpha(0),\beta(0)]$ be fixed. Since, by assumption,
	  $\rho$ is increasing on $\R$ and $\alpha,\beta$ satisfy \eqref{eq.assumptionlowerupper}, 
	  we have $\rho(\nu)\in [\alpha(T),\beta(T)]$; as a consequence,
	  by the existence result in Theorem \ref{thm.mainconcrete}, the Dirichlet problem
	  $$
	  (\mathrm{D}_{\nu}) \qquad\quad
	  \begin{cases}
	   \dsy\Big(\Phi\big(a(t,x(t))\,x'(t)\big)\Big)' = f(t,x(t),x'(t)), \qquad \text{a.e.\,on $I$}, \\[0.2cm]
	   x(0) = \nu,\,\,x(T) = \rho(\nu)
	  \end{cases}
	  $$ 
	  admits one solution $x_\nu$
	  such that $\alpha\leq x_\nu\leq \beta$ on $I$. 
	  Moreover, if $M > 0$
	  is such that $\sup_I|\alpha|,\,\sup_I|\beta|\leq M$
	  and $L_M > 0$ is as in Theorem \ref{thm.mainconcrete}, we have
	  $$(\ast) \qquad\quad
	  \text{$\sup_{t\in I}|x_\nu(t)|\leq M$ \quad and \quad 
	  $\sup_{t\in I}|\mathcal{A}_{x_\nu}(t)|\leq L_M$}.$$
       We then consider the following set:
	  \begin{align*}
	   & V := \Big\{\text{$\nu\in[\alpha(0),\beta(0)]$\,:\,\,$\exists$ a solution
	   $x_\nu\in W^{1,p}(I)$ of $(\mathrm{D}_\nu)$ s.t.\,$\alpha\leq x_\nu\leq\beta$,} \\
	   & \qquad\qquad\quad \text{$x_\nu$ satisfies $(\ast)$ and
	   $g(x_\nu(0),x_\nu(T),\mathcal{A}_{x_\nu}(0),\mathcal{A}_{x_\nu}(T))\geq 0$}\Big\}.
	  \end{align*}
	  
	   \textsc{Claim I:} We have $\nu := \alpha(0)\in V$.
	   In fact, by Theorem \ref{thm.mainconcrete}, there exists
	   a solution $x_{\nu}\in W^{1,p}(I)$ 
	   of $(\mathrm{D}_{\nu})$ such that 
	   $$\text{$\alpha\leq x_{\nu}\leq \beta$ on $I$},$$
	   and satisfying $(\ast)$; in particular, we have
	   $$\text{$x_{\nu}(0) = \nu = \alpha(0)$.} $$
	   From this, 
	   by applying Lemma
	   \ref{lem.technicalAx} with $x_\nu$ in place of $\beta$
	   (notice that, $x_\nu$ be\-ing a solution of $(\mathrm{D}_{\nu})$,
	   it is also a upper solution of \eqref{eq.generalODEgeneralBVP}), we get
	   $$\mathcal{A}_\alpha(0)\leq \mathcal{A}_{x_\nu}(0).$$
	   Analogously, since we have (remind that $\alpha$ satisfies \eqref{eq.assumptionlowerupper})
	   $$\text{$x_{\nu}(T) = \rho(\nu)
	   = \rho(\alpha(0)) = \alpha(T)$},$$
	   again by Lemma \ref{lem.technicalAx} (with $\beta$ replaced by $x_{\nu}$) we have
	   $$\mathcal{A}_{\alpha}(T)\geq \mathcal{A}_{x_{\nu}}(T).$$
	   Thus, by
	   \eqref{eq.assumptionlowerupper} and
	    assumption (G2) we obtain
	   $$g(x_\nu(0),x_\nu(T),\mathcal{A}_{x_\nu}(0),\mathcal{A}_{x_\nu}(T))
	   \geq g(\alpha(0),\alpha(T),\mathcal{A}_\alpha(0),\mathcal{A}_\alpha(T))\geq 0.$$
	   This proves that $\nu\in V$, as claimed. In particular, $V\neq \varnothing$. \medskip
	   
	   \textsc{Claim II:} If $\overline{\nu} := \sup V$, we have $\overline{\nu}\in V$.
	   In fact, if $\overline{\nu} = \alpha(0)$, we have already proved
	   in Claim I that $\overline{\nu}\in V$; if, instead, $\overline{\nu} > \alpha(0)$,
	   we choose a sequence $\{\nu_n\}_n\subseteq V$ such that
	   $\nu_n\to\overline{\nu}$ as $n\to\infty$ and $\nu_n\leq\overline{\nu}$ for every $n$.
	   Since $\{\nu_n\}_n\subseteq V$, there exists a solution
	   $x_n\in W^{1,p}(I)$ of $\mathrm{(D_{\nu_n})}$ s.t. \medskip
	   
	   (a)\,\,$\alpha\leq x_n\leq\beta$ on $I$; \medskip
	   
	   (b)\,\,$x_n$ satisfies $(\ast)$; \medskip
	   
	   (c)\,\,$g(x_n(0),x_n(T),\mathcal{A}_{x_n}(0),\mathcal{A}_{x_n}(T))\geq 0$. \medskip
	   
	   On account of (b) we can apply Proposition \ref{prop.compactness},
	   which provides us with a solution $x_0\in W^{1,p}(I)$ of \eqref{eq.generalODEgeneralBVP}
	   such that (up to a sub-sequence)
	   $$x_n(t)\to x_0(t) \quad \text{and}
	   \quad \mathcal{A}_{x_n}(t)\to\mathcal{A}_{x_0}(t) \qquad \text{for every $t\in I$}.$$
	   Now, since $\nu_n\to\overline{\nu}$ and $\rho$ is continuous, it is readily
	   seen that $x_0$ is a solution of $\mathrm{(D_{\overline{\nu}})}$;
	   moreover, since $x_n$ satisfies $(\ast)$ 
	   and $\alpha\leq x_n\leq\beta$ on $I$ for every $n\in\N$,
	   then the same is true of $x_0$. Finally, by (c) and the continuity of
	   the function $g$ on whole of $\R^4$ (see assumption (G2))
	   we conclude that
	   \begin{equation*}
	    \begin{split}
	    & g(x_0(0),x_0(T),\mathcal{A}_{x_0}(0),\mathcal{A}_{x_0}(T)) \\[0.2cm]
	   & \quad = \lim_{n\to\infty}g(x_n(0),x_n(T),\mathcal{A}_{x_n}(0),\mathcal{A}_{x_n}(T)) \geq 0,
	   \end{split}
	   \end{equation*}
		and this proves that $\overline{\nu}\in V$. \medskip
		
	  \noindent  With Claims I and II at hand, we now prove
	  the existence of a solution
	  for \eqref{eq.maingeneralBVPgh}. In fact, let $\overline{\nu} = \sup V
	  \in [\alpha(0),\beta(0)]$
	  and let $x_{\overline{\nu}}\in W^{1,p}(I)$ be a cor\-re\-spon\-ding 
	  so\-lu\-tion
	  of $\mathrm{(D_{\overline{\nu}})}$ satisfying $(\star)$ and such that
	  \begin{itemize}
	   \item[(i)] $\alpha\leq x_{\overline{\nu}}\leq \beta$ on $I$;
	   \item[(ii)] $g(x_{\overline{\nu}}(0),x_{\overline{\nu}}(T),
	   \mathcal{A}_{x_{\overline{\nu}}}(0),\mathcal{A}_{x_{\overline{\nu}}}(T))\geq 0$.
	  \end{itemize}
	  If $\overline{\nu} = \beta(0)$, we have $x_{\overline{\nu}}(0) = \beta(0)$ and
	  (by \eqref{eq.assumptionlowerupper})
	  $$x_{\overline{\nu}}(T) = 
	  \rho(\beta(0)) = \beta(T);$$ 
	  on the other hand,
      since we also know that $x_0\leq \beta$ on $I$, 
	  from Lemma \ref{lem.technicalAx}
	  (with $\alpha$ replaced by $x_{\overline{\nu}}$, which is a solution
	  of $\mathrm{(D_{\overline{\nu}})}$) 
	  we infer that
	  $$\text{$\mathcal{A}_{x_{\overline{\nu}}}(0)\leq\mathcal{A}_\beta(0)$
	  and $\mathcal{A}_{x_{\overline{\nu}}}(T)\geq\mathcal{A}_\beta(T)$}.$$ 
	  Hence,
	  by (ii), the monotonicity of $g$ (see (G2)), and \eqref{eq.assumptionlowerupper}
	  we obtain
	  \begin{align*}
	   & 0 \leq g(x_{\overline{\nu}}(0),x_{\overline{\nu}}(T),
	   \mathcal{A}_{x_{\overline{\nu}}}(0),\mathcal{A}_{x_{\overline{\nu}}}(T))
	   = g(\beta(0),\beta(T),\mathcal{A}_{x_{\overline{\nu}}}(0),
	   \mathcal{A}_{x_{\overline{\nu}}}(T)) \\[0.1cm]
	   & \qquad\qquad \leq g(\beta(0),\beta(T),\mathcal{A}_{\beta}(0),\mathcal{A}_{\beta}(T)) \leq 0,
	  \end{align*}
	  and this proves that $x_{\overline{\nu}}$ is a solution of \eqref{eq.maingeneralBVPgh}
	  satisfying \eqref{eq.solutionxfunctionalinterval}
	  and \eqref{eq.estimsolMLgeneralgh}. \medskip
	  
	  If, instead $\overline{\nu} < \beta(0)$, we choose a sequence $\{\mu_m\}_m\subseteq[\alpha(0),\beta(0)]$
	  such that $\mu_m\to\overline{\nu}$ as $m\to\infty$
	  and $\mu_m > \overline{\nu}$ for every $m$.
	  Since $x_{\overline{\nu}}$ 
	  is a solution of $\mathrm{(D_{\overline{\nu}})}$ satisfying
	  (i)-(ii) above, we can think of $x_{\overline{\nu}}$ 
	  and $\beta$ as, respectively,
	  a lower and a upper solution of \eqref{eq.generalODEgeneralBVP} satisfying
	  (A1') in Theorem \ref{thm.mainconcrete}; moreover, by the very
	  choice of $M > 0$ we also have that
	  $$\sup_{t\in I}|x_{\overline{\nu}}(t)|,\,\,\sup_{t\in I}|\beta(t)|\leq M.$$
	  Hence, for every $m$ there exists a solution
	  $u_m$ of $\mathrm{(D_{\mu_m})}$ such that \medskip
	  
	  $\bullet$\,\,$\alpha\leq x_{\overline{\nu}}
	  \leq u_m\leq\beta$ on $I$; \medskip
	  
	  $\bullet$\,\,$\sup_I|u_m|\leq M$ and $ \sup_I|\mathcal{A}_{u_m}| \leq L_M$. \medskip
	  
	  \noindent In particular, $u_m$ satisfies $(\ast)$ for any $m$.
	  We can then apply Proposition \ref{prop.compactness},
	   which provides us with a solution $u_0$ of \eqref{eq.generalODEgeneralBVP} such that
	   (up to a sub-sequence)
		$$u_m(t)\to u_0(t) \quad \text{and}
	   \quad \mathcal{A}_{u_m}(t)\to\mathcal{A}_{u_0}(t) \qquad \text{for every $t\in I$}.$$
        Since $\mu_m\to\overline{\nu}$ and $\rho$ is continuous,
		$u_0$ solves $\mathrm{(D_{\overline{\nu}})}$; hence
		$$u_0(T) = \rho(u_0(0)).$$
		We now observe that, since $\mu_m > \overline{\nu} = \sup V$, then $\mu_m\notin V$;
		as a consequence, since $\alpha\leq u_m\leq \beta$ on $I$
		and $u_m$ satisfies $(\ast)$ for every $m$, we necessarily have
		$$g(u_m(0),u_m(T),\mathcal{A}_{u_m}(0),\mathcal{A}_{u_m}(T)) < 0.$$
		From this, by the continuity of $g$ (see assumption (G1)), we get
		\begin{equation} \label{eq.touseguzero}
		 g(u_0(0),u_0(T),\mathcal{A}_{u_0}(0),\mathcal{A}_{u_0}(T)) \leq 0.
		\end{equation}
		On the other hand, since both $x_{\overline{\nu}}$ 
		and $u_0$ solve $\mathrm{(D_{\overline{\nu}})}$,
		we have 
		$$\text{$u_0(0) = x_{\overline{\nu}}(0) = 
		\overline{\nu}$ and $u_0(T) = x_{\overline{\nu}}(T) = \rho(\overline{\nu})$};$$
		moreover, since $u_m\geq x_{\overline{\nu}}$ on $I$ for every $m$, 
		then the same is true of $u_0$. 
		From this, by exploiting once again Lemma \ref{lem.technicalAx} 
	    (with $\alpha = x_{\overline{\nu}}$ and $\beta = u_m$, which are
	    solutions of $\mathrm{(D_{\overline{\nu}})}$), we infer that
		$$\mathcal{A}_{u_0}(0)\geq \mathcal{A}_{x_0}(0) \quad \text{and}
		\quad \mathcal{A}_{u_0}(T)\leq \mathcal{A}_{x_0}(T).$$
		By \eqref{eq.touseguzero}, the by monotonicity of $g$
		and (ii) above, we then obtain
		\begin{align*}
		 & 0 \geq g(u_0(0),u_0(T),\mathcal{A}_{u_0}(0),\mathcal{A}_{u_0}(T))
		 = g(x_{\overline{\nu}}(0),x_{\overline{\nu}}(T),
		 \mathcal{A}_{u_0}(0),\mathcal{A}_{u_0}(T)) \\[0.1cm]
		 & \qquad\qquad\geq g(x_{\overline{\nu}}(0),x_{\overline{\nu}}(T), 
		 \mathcal{A}_{x_{\overline{\nu}}}(0), \mathcal{A}_{x_{\overline{\nu}}}(T)) \geq 0, 
		\end{align*}
		and this shows that $u_0$ solves \eqref{eq.maingeneralBVPgh}.
		Finally, since $\alpha\leq u_m\leq \beta$ on $I$ and $u_m$ satisfies $(\ast)$
		for every $m$, we conclude that $u_0$ satisfies 
		\eqref{eq.solutionxfunctionalinterval}-\eqref{eq.estimsolMLgeneralgh}. 
	  \end{proof}
	  As a particular case of Theorem \ref{thm.maingeneralBVP}, we have
	  the following result.
	  \begin{cor} \label{cor.existencePeriodic}
	   Let us assume that \emph{all} the hypotheses of Theorem \ref{thm.mainconcrete}
	   are satisfied; moreover, 
	   if $\alpha,\beta \in W^{1,p}(I)$ are as in assumption \emph{(A1')}, we suppose that
	   the following inequality hold:
	\begin{equation*}
	 \begin{cases}
	 \mathcal{A}_\alpha(0) \geq \mathcal{A}_\alpha(T), \\[0.2cm]
	 \alpha(T) = \alpha(0),
	 \end{cases} \quad
	 \begin{cases}
	 \mathcal{A}_\beta(0) \leq \mathcal{A}_\beta(T), \\[0.2cm]
	 \beta(T) = \beta(0).
	 \end{cases}
	\end{equation*}
	Finally, let us assume that the function
	$a$ satisfies the
	condition:
	$$a(0,x) \neq 0 \quad \text{and} \quad a(T,x) \neq 0 \qquad \text{for every $x\in\R$}.$$
	Then, there exists 
	\emph{(}at least\emph{)} one solution $x\in W^{1,p}(I)$ of
	$$
	\begin{cases}
	 \dsy \dsy\Big(\Phi\big(a(t,x(t))\,x'(t)\big)\Big)' = 
	 f(t,x(t),x'(t)), \qquad \text{a.e.\,on $I$}, \\[0.2cm]
	 \mathcal{A}_x(0) = \mathcal{A}_x(T), \\[0.2cm]
	 x(0) = x(T).
	\end{cases}
	$$
	  \end{cor}
	\begin{proof}
	 It is a straightforward consequence of Theorem \ref{thm.maingeneralBVP}
	 applied to 
	 $$\rho(r) = r \quad \text{and} \quad g(u,v,w,z) = w-z$$
	 (which trivially satisfy assumptions (G1) and (G2)). This ends the proof.
	\end{proof}
	\medskip
	We conclude the present section by  turning our attention
	to Sturm-Lio\-uvil\-le-type and Neumann-type problems
	associated with the ODE \eqref{eq.generalODEgeneralBVP}.
	
	\noindent To be more precise, we consider the following boundary value problems:
	\begin{equation} \label{eq.maingeneralBVPpq}
	 \begin{cases}
	 \dsy \dsy\Big(\Phi\big(a(t,x(t))\,x'(t)\big)\Big)' = 
	 f(t,x(t),x'(t)), \qquad \text{a.e.\,on $I$}, \\[0.2cm]
	 p(x(0),\mathcal{A}_x(0)) = 0,\,\,q(x(T),\mathcal{A}_x(T)) = 0.
	 \end{cases}
	\end{equation}
	Here, the functions $p,q:\R^2\longto\R$ satisfies the following general
	assumptions:
	\begin{itemize}
	 \item[(S1)] $p\in C(\R^2,\R)$ and, for every $s\in\R$, the map
	 $p(s,\cdot)$ is increasing on $\R$; \medskip
	 \item[(S2)] $q\in C(\R^2,\R)$ and, for every $s\in\R$, the map
	 $q(s,\cdot)$ is decreasing on $\R$.
	\end{itemize}
	The following theorem is the second main result of this section.
	\begin{thm} \label{thm.maingeneralBVPBIS}
	Let us assume that \emph{all} the hypotheses of Theorem \ref{thm.mainconcrete} are
	satisfied, and that the functions $p$ and $q$ fulfill the assumptions \emph{(S1)-(S2)}
	introduced above.	
	Moreover, if $\alpha,\beta \in W^{1,p}(I)$ are as in assumption \emph{(A1')}, we suppose that
	the following inequality hold:
	\begin{equation} \label{eq.assumptionlowerupperSturm}
	 \begin{cases}
	 p(\alpha(0),\mathcal{A}_\alpha(0)) \geq 0, \\[0.2cm]
	 q(\alpha(T),\mathcal{A}_\alpha(T))\geq 0;
	 \end{cases} \quad
	 \begin{cases}
	 p(\beta(0),\mathcal{A}_\beta(0)) \leq 0, \\[0.2cm]
	 q(\beta(T),\mathcal{A}_\beta(T))\leq 0.
	 \end{cases}
	\end{equation}
	Finally, let us assume that $a$ satisfies the
	``compatibility'' condition:
	$$a(0,x) \neq 0 \quad \text{and} \quad a(T,x) \neq 0 \qquad \text{for every $x\in\R$}.$$
	Then the problem \eqref{eq.maingeneralBVPpq}
	possesses one solution $x\in W^{1,p}(I)$ such that
	\begin{equation} \label{eq.solutionxfunctionalintervalSturm}
	 \alpha(t)\leq x(t)\leq\beta(t) \qquad\text{for every $t\in I$}.
	\end{equation}
	Furthermore, if $M > 0$ is any real number
	such that $\sup_I|\alpha|,\,\sup_I|\beta|\leq M$
	and $L_M > 0$ is as in Theorem \ref{thm.mainconcrete}
	\emph{(}see \eqref{eq.choiceL2}\emph{)},
	then
   \begin{align}
    \max_{t\in I}\big|x(t)\big|\leq M \quad \text{\emph{and}} \quad
    \max_{t\in I}\big|\mathcal{A}_x(t)\big| \leq L. \label{eq.estimsolMLgeneralpq}
   \end{align}
	\end{thm}
	The proof of Theorem \ref{thm.maingeneralBVPBIS} relies
	on the following lemma.
	\begin{lem} \label{lem.prelimgeneralBVPBIS}
	 Let the assumptions and the notation of Theorem \ref{thm.maingeneralBVPBIS}
     apply. 
     Then, for every fixed $\nu\in[\alpha(T),\beta(T)]$, the boundary value problem
     $$
      \mathrm{(D_\nu)} \qquad\quad
      \begin{cases}
       \dsy\Big(\Phi\big(a(t,x(t))\,x'(t)\big)\Big)' = f(t,x(t),x'(t)), \qquad \text{a.e.\,on $I$}, \\[0.2cm]
       p(x(0),\mathcal{A}_x(0)) = 0, \\[0.2cm]
       x(T) = \nu.
      \end{cases}
     $$
     possesses at least one solution $x\in W^{1,p}(I)$ such that
     $\alpha\leq x\leq \beta$ on $I$.
     Fur\-ther\-more, if $M$ and $L_M$ are as in the statement of Theorem \ref{thm.maingeneralBVPBIS},
     then
     \begin{equation} \label{eq.estimLemmageneralpq}
      \sup_{t\in I}|x(t)|\leq M \quad \text{and}\quad
     \sup_{t\in I}|\mathcal{A}_x(t)|\leq L_M.
     \end{equation}
	\end{lem}
	\begin{proof}
	 We fix $\nu\in[\alpha(T),\beta(T)]$ and we consider the following functions: \medskip
     
     $(\star)$\,\,$\rho:\R\to\R, \quad \rho(r) := \nu$; \medskip
     
     $(\star)$\,\,$g:\R^4\to\R, \quad g(u,v,w,z) := p(u,w)$. \medskip
     
     \noindent Then, by means of these functions, 
     we can re-write the problem $(\mathrm{D}_\nu)$ as
     $$
      \begin{cases}
       \dsy\Big(\Phi\big(a(t,x(t))\,x'(t)\big)\Big)' = f(t,x(t),x'(t)), \qquad \text{a.e.\,on $I$}, \\[0.2cm]
       g(x(0),x(T),\mathcal{A}_x(0),\mathcal{A}_x(T)) = 0, \\[0.2cm]
       x(T) = \rho(x(0)).
      \end{cases}
     $$
     Now, taking into account $(\mathrm{S1})$, it is readily
     seen that $\rho$ and $g$ satisfy, re\-spec\-ti\-ve\-ly, (G1) and (G2)
     in the statement of Theorem \ref{thm.maingeneralBVP}; thus, to prove the lemma,
     it suffices to show the existence 
     of a lower and a upper solution 
     for \eqref{eq.generalODEgeneralBVP} satisfying
     (A1') and \eqref{eq.assumptionlowerupper}
     (with the above choices of $\rho$ and $g$). \medskip
     
     To this end we first observe that, by assumption,
     $\alpha$ and $\beta$ are, re\-spec\-ti\-ve\-ly, a lower 
     and a upper solution for \eqref{eq.generalODEgeneralBVP} 
     satisfying (A1') (that is,
     $\alpha\leq\beta$ on $I$); as a consequence, by Theorem \ref{thm.mainconcrete}, 
	the Dirichlet problem
	$$
	\mathrm{(D)_1}\qquad\quad
	\begin{cases}
       \dsy\Big(\Phi\big(a(t,x(t))\,x'(t)\big)\Big)' = f(t,x(t),x'(t)), \qquad \text{a.e.\,on $I$}, \\[0.2cm]
       x(0) = \alpha(0),\,\,x(T) = \nu
	\end{cases}
	$$
	possesses (at least) one solution $x_1\in W^{1,p}(I)$ such that
	$\alpha\leq x_1\leq \beta$ on $I$. Moreover,
	if $M$ and $L_M$ are as in the statement of Theorem \ref{thm.maingeneralBVP},
	we have
		\begin{equation} \label{eq.estimx1lowersol}
	 \sup_{t\in I}|x_1(t)| \leq M \quad \text{and} 
	 \quad \sup_{t\in I}|\mathcal{A}_{x_1}(t)| \leq L_M.
	\end{equation}
	We claim that the function $x_1$, which is obviously a lower
	solution of \eqref{eq.generalODEgeneralBVP}, satisfies
	the first assumption in \eqref{eq.assumptionlowerupper}
	(with $g$ as above). In fact,
	since 
	$$\text{$x_1(0) = \alpha(0)$ and $x_1\geq \alpha$ on $I$,}$$
	from Lemma \ref{lem.technicalAx} (with $x_1$ in place of $\beta$) we infer that
	$$\mathcal{A}_{x_1}(0)\geq \mathcal{A}_\alpha(0);$$ 
	as a consequence, by 
	assumption (S1) and \eqref{eq.assumptionlowerupperSturm}, we obtain
	\begin{align*}
    & g(x_1(0),x_1(T),\mathcal{A}_{x_1}(0),\mathcal{A}_{x_1}(T))
    = p(x_1(0),\mathcal{A}_{x_1}(0)) \\[0.1cm]
    & \qquad\quad = p(\alpha(0),\mathcal{A}_{x_1}(0)) 
    \geq p(\alpha(0),\mathcal{A}_\alpha(0)) \geq 0.
	\end{align*}
	Furthermore, since $x_1$ solves $\mathrm{(D)_1}$, we have $x_1(T) = 
	\nu = \rho(x_1(0))$, 
	and this proves
	that $x_1$ satisfies the first assumption in \eqref{eq.assumptionlowerupper}.

	 We now turn to prove the existence of a upper solution $x_2$ of
	 \eqref{eq.generalODEgeneralBVP} such that $x_2\geq x_1$ on $I$
	 and satisfying the second assumption in \eqref{eq.assumptionlowerupper}.

	 First of all, we notice that $x_1$ and $\beta$ are, respectively, 
	 a lower and a upper solution for \eqref{eq.generalODEgeneralBVP}
	 such that $x_1\leq\beta$ on $I$; moreover, 
	 $$\nu = x_1(T)\in[x_1(T),\beta(T)].$$
	 Finally, by
	 \eqref{eq.estimx1lowersol} and the choice of $M$, we have
	 $$\sup_{t\in I}|x_1(t)|,\,\,\sup_{t\in I}|\beta(t)|\leq M.$$
	 As a consequence, by Theorem \ref{thm.mainconcrete}, the Dirichlet problem
	 $$\mathrm{(D)_2}\qquad\quad
	\begin{cases}
       \dsy\Big(\Phi\big(a(t,x(t))\,x'(t)\big)\Big)' = f(t,x(t),x'(t)), \qquad \text{a.e.\,on $I$}, \\[0.2cm]
       x(0) = \beta(0),\,\,x(T) = \nu	
	\end{cases}$$
	 has a solution $x_2\in W^{1,p}(I)$ such that 
	 $x_1\leq x_2\leq \beta$ on $I$, further satisfying
	 \begin{equation} \label{eq.estimx2uppersol}
	  \sup_{t\in I}|x_2(t)| \leq M \quad \text{and} \quad 
	  \sup_{t\in I}|\mathcal{A}_{x_2}(t)| \leq L_M,
	 \end{equation}
	 for the same $M > 0$ fixed at beginning (and $L_M$ as in Theorem
	 \ref{thm.mainconcrete}). We claim that $x_2$, which is obviously
	 a upper solution of \eqref{eq.generalODEgeneralBVP}, satisfies
	 the second assumption in \eqref{eq.assumptionlowerupper}. 
	 In fact, 
	 since 
	 $$\text{$x_2(0) = \beta(0)$ and $x_2\leq\beta$ on $I$},$$ 
	 by exploiting
	 Lemma \ref{lem.technicalAx} (with $x_2$ in place of $\alpha$) we get
	 $$\mathcal{A}_{x_2}(0)\leq\mathcal{A}_{\beta}(0);$$ 
	 thus, by
	 assumption (G2) and again \eqref{eq.assumptionlowerupperSturm} we have
	\begin{align*}
    & g(x_2(0),x_2(T),\mathcal{A}_{x_2}(0),\mathcal{A}_{x_2}(T))
    = p(x_2(0),\mathcal{A}_{x_2}(0)) \\[0.1cm]
    & \qquad\quad = p(\beta(0),\mathcal{A}_{x_2}(0)) 
    \leq p(\beta(0),\mathcal{A}_\beta(0)) \leq 0.
	\end{align*}
	Furthermore, since $x_2$ solves $(\mathrm{D}_2)$, one has $x_2(T) = \nu
	= \rho(x_2(0)) $,
	and this proves that $x_2\geq x_1$ satisfies the second assumption
	in \eqref{eq.assumptionlowerupperSturm}. \medskip
	
	Gathering together all these facts, we can conclude that
	all the as\-sum\-ptions in Theorem \ref{thm.maingeneralBVP} are
	fulfilled
	(with the above choice of $g$ and $\rho$); 
    thus,
	there exists (at least)
	one solution $x\in W^{1,p}(I)$ of $\mathrm{(D_\nu)}$ such that
	$$\alpha\leq x_1\leq x\leq x_2\leq \beta \quad \text{on $I$}.$$  
	In particular, by the choice of $M$ and $L_M$ (according 
	to \eqref{eq.choiceL2})
	and the fact that $x_1,\,x_2$ fulfill, respectively, 
    \eqref{eq.estimx1lowersol} and
	\eqref{eq.estimx2uppersol}, we deduce that
	$$\sup_{t\in I}|x(t)|\leq M \quad \text{and}\quad
     \sup_{t\in I}|\mathcal{A}_x(t)|\leq L_M,$$
	with \emph{the very same $M, L_M > 0$ fixed at the beginning}.
	This ends the proof.  	
	\end{proof}
	\begin{rem} \label{rem.assumptionsalphabetap}
	 Let the assumptions and the notation
	 of Lemma \ref{lem.prelimgeneralBVPBIS} apply.

	 By giving a closer inspection to the proof
	 of this lemma, we see that the only property
	 of $\alpha$ and $\beta$ we have used is the following (see \eqref{eq.assumptionlowerupperSturm}):
	 $$\big(\bigstar\big)\qquad\quad
	 p(\alpha(0),\mathcal{A}_\alpha(0))\geq 0 \qquad\text{and}\qquad
	 p(\beta(0),\mathcal{A}_\beta(0))\leq 0.$$
	 Hence, Lemma \ref{lem.prelimgeneralBVPBIS} still holds
	 if we replace \eqref{eq.assumptionlowerupperSturm} with the weaker $\big(\bigstar\big)$.
	\end{rem}
	Thanks to Lemma \ref{lem.prelimgeneralBVPBIS}, 
	we are able to prove Theorem \ref{thm.maingeneralBVPBIS}.
	\begin{proof} [Proof (of Theorem \ref{thm.maingeneralBVPBIS})]
    Let $\nu\in[\alpha(T),\beta(T)]$ be fixed. 
    By Lemma \ref{lem.prelimgeneralBVPBIS}, there exists (at least)
    one solution $x_\nu\in W^{1,p}(I)$ of the BVP
    $$
	 \mathrm{(D_\nu)} \qquad\quad
	 \begin{cases}
	  \dsy\Big(\Phi\big(a(t,x(t))\,x'(t)\big)\Big)' = f(t,x(t),x'(t)), \qquad \text{a.e.\,on $I$}, \\[0.2cm]
       p(x(0),\mathcal{A}_x(0)) = 0, \\[0.2cm]
       x(T) = \nu,
	 \end{cases}
    $$
	such that $\alpha\leq x_\nu\leq \beta$ on $I$; moreover,
	if $M,L_M > 0$ are as in the statement of the theorem
	(that is, $L_M$ is as in \eqref{eq.choiceL2}), we have
	(see Lemma \ref{lem.prelimgeneralBVPBIS})
	$$
	\big(\ast)\qquad\quad
	\sup_{t\in I}|x_{\nu}(t)|\leq M \quad \text{and} \quad
	\sup_{t\in I}|\mathcal{A}_{x_\nu}(t)|\leq L_M.$$
	We then consider the following set:
	\begin{align*}
	 & V := \Big\{\text{$\nu\in[\alpha(T),\beta(T)]$\,:\,\,$\exists$ a solution
	   $x_\nu\in W^{1,p}(I)$ of $(\mathrm{D}_\nu)$ s.t.\,$\alpha\leq x_\nu\leq\beta$,} \\
	   & \qquad\qquad\qquad \text{$x_\nu$ satisfies $(\ast)$ and
	   $q(x_\nu(T),\mathcal{A}_{x_\nu}(T))\geq 0$}\Big\}.
	\end{align*}
	
	\textsc{Step I:} We have $\nu = \alpha(T)\in V$. In fact,
	by Lemma \ref{lem.prelimgeneralBVPBIS}, there exists a
    solution $x_\nu\in W^{1,p}(I)$ of $\mathrm{(D_\nu)}$ such that
	$\alpha\leq x_\nu\leq \beta$ and satisfying $(\ast)$; in particular,
	we have $x_\nu(T) = \alpha(T)$.
	Hence, by Lemma \ref{lem.technicalAx} we get
	$$\mathcal{A}_{x_\nu}(T) \leq \mathcal{A}_\alpha(T).$$
	From this, by assumption (S2) and \eqref{eq.assumptionlowerupperSturm}, we obtain
	$$q(x_\nu(T),\mathcal{A}_{x_\nu}(T)) 
	= q(\alpha(T),\mathcal{A}_{x_{\nu}}(T)) \geq q(\alpha(T),\mathcal{A}_\alpha(T)) \geq 0.$$
	This proves that $\nu\in V$, as claimed. In particular, $V\neq\varnothing$. \medskip
	
	\textsc{Step II:} Setting $\overline{\nu} := \sup V$, we have
	$\overline{\nu}\in V$. In fact, if $\overline{\nu} = \alpha(T)$,
    by Step I we know that $\overline{\nu}\in V$;
	if, instead, $\overline{\nu} > \alpha(T)$, we can choose a sequence
	$\{\nu_n\}_n\subseteq V$ such that $\nu_n\to\overline{\nu}$
	as $n\to\infty$ and $\nu_n\leq \overline{\nu}$ for every $n$.
	Then, for every natural $n$ there exists
	a solution $x_n\in W^{1,p}(I)$ of $(\mathrm{D}_{\nu_n})$
	such that \medskip
	
	(a)\,\,$\alpha\leq x_n\leq \beta$ on $I$; \medskip
	
	(b)\,\,$x_n$ satisfies $(\ast)$; \medskip
	
	(c)\,\,$q(x_n(T),\mathcal{A}_{x_n}(T))\geq 0$.  \medskip
	
	\noindent On account of (b) we can apply Proposition
	\ref{prop.compactness}, which provides us with a solution 
	$x_0\in W^{1,p}(I)$ of \eqref{eq.generalODEgeneralBVP} such that,
	up to a sub-sequence,
	$$x_n(t)\to x_0(t) \quad \text{and}
	   \quad \mathcal{A}_{x_n}(t)\to\mathcal{A}_{x_0}(t) \qquad \text{for every $t\in I$}.$$
	Now, since $\nu_n\to\overline{\nu}$ and $p$ is continuous, we see that
	$x_0$ solves $\mathrm{(D_{\overline{\nu}})}$; hence
	$$p(x_0(0),\mathcal{A}_{x_0}(0)) = 0.$$
	Moreover,
	since $x_n$ satisfies $(\ast)$ 
	and $\alpha\leq x_n\leq \beta$ on $I$ for every natural $n$, then
 	 the same is true of $x_0$. Finally, by (c)
 	 and the continuity of $q$, one has
	\begin{equation*} 
	 q(x_0(T),\mathcal{A}_{x_0}(T)) =
	\lim_{n\to\infty}q(x_n(T),\mathcal{A}_{x_n}(T)) \geq 0.
	\end{equation*}
	This proves that $\overline{\nu}\in V$, as claimed. \medskip
	
	Now we have established Claims I and II, we can finally prove
	the existence of a solution to
	\eqref{eq.maingeneralBVPpq}. In fact, let $\overline{\nu}
	= \sup V \in [\alpha(T),\beta(T)]$,
	and let $x_{\overline{\nu}}\in W^{1,p}(I)$ be 
	a solution of $\mathrm{(D_{\overline{\nu}})}$
	satisfying $(\star)$ and such that \medskip
	
	(i) $\alpha\leq x_{\overline{\nu}}\leq \beta$ on $I$; \medskip
	
	(ii) $q(x_{\overline{\nu}}(T),\mathcal{A}_{x_{\overline{\nu}}}(T))\geq 0$. \medskip
	
	\noindent If $\overline{\nu} = \beta(T)$, we have 
	$$\text{$x_{\overline{\nu}}(T) = \overline{\nu} = \beta(T)$\,\,
	and \,\, $p(x_{\overline{\nu}}(0),\mathcal{A}_{x_{\overline{\nu}}}(0)) = 0$};$$ 
	moreover, since
	$x_{\overline{\nu}}\leq\beta$ on $I$, Lemma \ref{lem.technicalAx} implies that
	$\mathcal{A}_{x_{\overline{\nu}}}(T)\geq \mathcal{A}_{\beta}(T)$;
	from this, by (ii),
	the monotonicity of $q$ (see (S2)) and \eqref{eq.assumptionlowerupperSturm},
	we obtain
	$$0\leq q(x_{\overline{\nu}}(T),\mathcal{A}_{x_{\overline{\nu}}}(T)) 
	= q(\beta(T),\mathcal{A}_{x_{\overline{\nu}}}(T))
	\leq q(\beta(T),\mathcal{A}_\beta(T))\leq 0,$$
	and this proves that $x_{\overline{\nu}}$ is a solution
	of \eqref{eq.maingeneralBVPpq} satisfying \eqref{eq.solutionxfunctionalintervalSturm}
	and \eqref{eq.estimsolMLgeneralpq}. \medskip
	
	If, instead, $\overline{\nu} < \beta(T)$, we choose a sequence $\{\mu_m\}_m\subseteq[\alpha(T),\beta(T)]$
	such that $\mu_m\to\overline{\nu}$ as $m\to\infty$ and $\mu_m > \overline{\nu}$
	for any $m$. Since $x_{\overline{\nu}}$ solves $\mathrm{(D_{\overline{\nu}})}$
    and
	$x_{\overline{\nu}}\leq\beta$ on $I$, we can think of $
	x_{\overline{\nu}}$ and $\beta$ as, respectively,
	a lower and a upper solution of \eqref{eq.generalODEgeneralBVP}
	satisfying (A1')
	in Theorem \ref{thm.mainconcrete} and
	$\big(\bigstar\big)$ in Remark \ref{rem.assumptionsalphabetap}
	(see indeed \eqref{eq.assumptionlowerupperSturm}); 
	moreover, by $(\ast)$ and the choice of $M$ we have
	$$\sup_{t\in I}|x_{\overline{\nu}}(t)|,\,\,\sup_{t\in I}|\beta(t)|\leq M.$$
	As a consequence, 
	by Remark \ref{rem.assumptionsalphabetap}, for every $m\in\N$ there exists a solution
	  $u_m\in W^{1,p}(I)$ of $\mathrm{(D_{\mu_m})}$ such that \medskip
	  
	  $\bullet$\,\,$\alpha\leq x_{\overline{\nu}}\leq u_m\leq\beta$ on $I$; \medskip
	  
	  $\bullet$\,\,$\sup_I|u_m|\leq M$ and $ \sup_I|\mathcal{A}_{u_m}| \leq L_M$. \medskip

	\noindent 
	In particular, $u_m$ satisfies $(\ast)$ for every $m$. We can then apply
	Proposition \ref{prop.compactness}, which provides us with a solution
	$u_0$ of \eqref{eq.generalODEgeneralBVP} such that (up to a sub-sequence)
	$$u_m(t)\to u_0(t) \quad \text{and}
	   \quad \mathcal{A}_{u_m}(t)\to\mathcal{A}_{u_0}(t) \qquad \text{for every $t\in I$}.$$
	Thus, since $\mu_m\to\overline{\nu}$ and $p$ is continuous,    
	we see that $u_0$ solves $\mathrm{(D_{\overline{\nu}})}$; hence
	$$p(u_0(0),\mathcal{A}_{u_0}(0)) = 0.$$	
	 We now observe that,
	since $\mu_m > \overline{\nu} = \sup V$, then $\mu_m\notin V$; as a consequence, since
	$\alpha\leq u_m\leq\beta$ on $I$
	and $u_m$ satisfies $(\ast)$ for every $m$,
	we necessarily have
	\begin{equation*}
	 q(u_m(T),\mathcal{A}_{u_m}(T)) < 0 \quad \text{for every $m\in\N$}
	\end{equation*}
	From this, by the continuity of $q$ (see assumption (S2)), we get
	\begin{equation} \label{eq.estimqsolu0}
	 q(u_0(T),\mathcal{A}_{u_0}(T)) = \lim_{m\to\infty}q(u_m(T),\mathcal{A}_{u_m}(T)) \leq 0.
	\end{equation}
	On the other hand, since both $x_{\overline{\nu}}$ and $u_0$ solve
	$\mathrm{(D_{\overline{\nu}})}$, we have
	$$x_{\overline{\nu}}(T) = \overline{\nu} = u_0(T);$$ 
	moreover,
	since $u_m\geq x_{\overline{\nu}}$ for every natural $m$
	(by the construction of $u_m$), then the same is true of $u_0$. From this,
	by using Lemma \ref{lem.technicalAx} we infer that
	$$\mathcal{A}_{u_0}(T) \leq \mathcal{A}_{x_{\overline{\nu}}}(T).$$
	By \eqref{eq.estimqsolu0}, the monotonicity of
	$q$ (see (S2)) and (ii) above, we then get
	\begin{align*}
	0 \geq q(u_0(T),\mathcal{A}_{u_0}(T)) = q(x_{\overline{\nu}}(T),\mathcal{A}_{u_0}(T))
	\geq q(x_{\overline{\nu}}(T),\mathcal{A}_{x_{\overline{\nu}}}(T)) \geq 0,
	\end{align*}
	and this shows that $u_0$ solves \eqref{eq.estimsolMLgeneralpq}. Finally, since
	$\alpha\leq u_m\leq \beta$ and $u_m$ sa\-ti\-sfi\-es $(\ast)$ for every $m$,
	we conclude that $u_0$ fulfills
	\eqref{eq.solutionxfunctionalintervalSturm}-\eqref{eq.estimsolMLgeneralpq}.  
	\end{proof}
	From Theorem \ref{thm.maingeneralBVPBIS} we easily deduce the following results.
	\begin{cor} \label{cor.solvabilitySturm}
	 Let us assume that \emph{all} the hypotheses of Theorem \ref{thm.mainconcrete}
	 are satisfied. Moreover, let $\ell_1,\,\ell_2,\,\nu_1,\,\nu_2\in\R$ and 
	 let $m_1,\,m_2\in [0,\infty)$. 
	 If $\alpha$ and $\beta$ are as in assumption
	 \emph{(A1)'}, we suppose that 
	 $$
	  \begin{cases}
	   \ell_1\,\alpha(0) + m_1\,\mathcal{A}_\alpha(0) \geq \nu_1, \\
	   \ell_2\,\alpha(T) - m_2\,\mathcal{A}_\alpha(T) \geq \nu_2; \\
	  \end{cases} \qquad
	  \begin{cases}
	   \ell_1\,\beta(0) + m_1\,\mathcal{A}_\beta(0) \leq \nu_1, \\
	   \ell_2\,\beta(T) - m_2\,\mathcal{A}_\beta(T) \leq \nu_2
	  \end{cases}
	 $$
	 Finally, we assume that the function $a$ fulfills the the following assumption:
	$$a(0,x) \neq 0 \quad \text{and} \quad a(T,x) \neq 0 \qquad \text{for every $x\in\R$}.$$
	 Then, there exists a solution
	 $x\in W^{1,p}(I)$ of the {Sturm-Liouville problem}
	 $$
	  \begin{cases}
	   \dsy \dsy\Big(\Phi\big(a(t,x(t))\,x'(t)\big)\Big)' = 
	   f(t,x(t),x'(t)), \qquad \text{a.e.\,on $I$}, \\[0.2cm]
	   \ell_1\,x(0) + m_1\,\mathcal{A}_x(0) = \nu_1, \\[0.1cm]
	   \ell_2\,x(T) - m_2\,\mathcal{A}_x(T) = \nu_2.
	  \end{cases}
	 $$
	\end{cor}
	\begin{proof}
	 It is a direct consequence of Theorem \ref{thm.maingeneralBVPBIS} applied to the functions
	 $$p(s,t) := \ell_1\,s + m_1\,t - \nu_1 \qquad \text{and}
	 \qquad
	 q(s,t) := \ell_2\,s - m_2\,t - \nu_2,$$
	 which satisfy (S1)-(S2) (since $m_1,m_2\geq 0$). This ends the proof.  
	\end{proof}
	\begin{cor} \label{cor.solvabilityNeumann}
	 Let us assume that \emph{all} the hypotheses of Theorem \ref{thm.mainconcrete}
	 are satisfied. Moreover, let $\nu_1,\,\nu_2\in\R$ be arbitrarily fixed.
	 If $\alpha$ and $\beta$ are as in assumption
	 \emph{(A1)'}, we suppose that the following conditions are satisfied:
	 $$
	  \begin{cases}
	   \mathcal{A}_\alpha(0) \geq \nu_1, \\
	   \mathcal{A}_\alpha(T) \leq \nu_2; \\
	  \end{cases} \qquad
	  \begin{cases}
	   \mathcal{A}_\beta(0) \leq \nu_1, \\
	   \mathcal{A}_\beta(T) \geq \nu_2
	  \end{cases}
	 $$
	 Finally, we assume that the function $a$ fulfills the following assumption:
	$$a(0,x) \neq 0 \quad \text{and} \quad a(T,x) \neq 0 \qquad \text{for every $x\in\R$}.$$
	 Then, there exists a solution
	 $x\in W^{1,p}(I)$ of the {Neumann problem}
	 $$
	  \begin{cases}
	   \dsy \dsy\Big(\Phi\big(a(t,x(t))\,x'(t)\big)\Big)' = f(t,x(t),x'(t)), 
	   \qquad \text{a.e.\,on $I$}, \\[0.2cm]
	   \mathcal{A}_x(0) = \nu_1, \\[0.1cm]
	   \mathcal{A}_x(T) = \nu_2.
	  \end{cases}
	 $$
	\end{cor}
	\begin{proof}
	 It is another direct consequence of Theorem \ref{thm.maingeneralBVPBIS} applied to 
	 $$p(s,t) := t - \nu_1 \qquad \text{and}
	 \qquad
	 q(s,t) := \nu_2 - t,$$
	 which obviously satisfy assumptions (S1)-(S2). This ends the proof.  
	\end{proof}


\end{document}